\documentclass[11pt,a4paper,reqno]{amsart}
\usepackage[includehead,includefoot,margin=25mm]{geometry}
\usepackage{times}
\usepackage{amsfonts} %
\usepackage{amsmath} %
\usepackage{amssymb} %
\usepackage{amsthm} %
\usepackage{enumerate} %
\usepackage{bbm} %
\usepackage{graphicx} %
\usepackage{nicefrac} %
\usepackage{color} %
\usepackage{hyperref}
\usepackage{tikz-cd}
\hypersetup{
    colorlinks=true,       % false: boxed links; true: colored links
    linkcolor=blue,          % color of internal links
    citecolor=blue,        % color of links to bibliography
    filecolor=blue,      % color of file links
    urlcolor=blue           % color of external links
}
\usepackage[all]{xy}
\usepackage{calrsfs} 

\usepackage{imakeidx}
\makeindex

\newtheorem{theorem}{Theorem}[section] %
\newtheorem{corollary}[theorem]{Corollary} %
\newtheorem{lemma}[theorem]{Lemma} %
\newtheorem*{mainthm}{Main Theorem} %
\newtheorem{proposition}[theorem]{Proposition} %
{\theoremstyle{remark} %
  \newtheorem{remark}[theorem]{Remark}} %
{\theoremstyle{definition} %
  \newtheorem{definition}[theorem]{Definition} %
  \newtheorem{example}[theorem]{Example} %
  \newtheorem{question}[theorem]{Question} %
}

\newcommand{\PP}[0]{\ensuremath{\mathbb{P}}}

\newcommand{\ZZ}[0]{\ensuremath{\mathbb{Z}}}

\newcommand{\OO}[0]{\ensuremath{\mathcal{O}}}

\def\fg{{\mathfrak g}}
\def\fg{{\mathfrak g}}\def\fh{{\mathfrak h}}
\def\fp{{\mathfrak p}}
\DeclareMathOperator{\Fl}{Fl}

\usepackage[normalem]{ulem}
\usepackage{dynkin-diagrams}
\usepackage{multirow}

\makeatletter
\let\@wraptoccontribs\wraptoccontribs
\makeatother

\begin{document}

\title[]{Projective manifolds whose tangent bundle is Ulrich}

\author{Pedro Montero}
\address{Departamento de Matem\'aticas, Universidad T\'ecnica
  Fe\-de\-ri\-co San\-ta Ma\-r\'\i a, Valpara\'\i
  so, Chile}  \email{pedro.montero@usm.cl}

\author{Yulieth Prieto--Monta\~{n}ez} %
\address{Dipartimento di Matematica, Universit\`a di Bologna, Piazza di Porta S. Donato, 5, I-40126 Bologna, Italy} %
\email{yulieth.prieto2@unibo.it}

\author{Sergio Troncoso}
\address{Departamento de Matem\'aticas, Universidad T\'ecnica
  Fe\-de\-ri\-co San\-ta Ma\-r\'\i a, Valpara\'\i
  so, Chile}  
\email{sergio.troncosoi@usm.cl}
  
\contrib[with an appendix by]{Vladimiro Benedetti}
\address{Institut de Mathématiques de Bourgogne, UMR CNRS 5584, Universit\'e de Bourgogne et Franche-Comt\'e, 9 Avenue Alain Savary, BP 47870, 21078 Dijon Cedex, France}  
\email{vladimiro.benedetti@u-bourgogne.fr}

\date{\today}

\thanks{
  \mbox{\hspace{11pt}}{\it Key words}: aCM vector bundles, Homogeneous spaces, Rationally connected varieties, Ulrich vector bundles.}

\begin{abstract}
In this article, we give numerical restrictions on the Chern classes of Ulrich bundles on higher-dimensional manifolds, which are inspired by the results of Casnati in the case of surfaces. As a by-product, we prove that the only projective manifolds whose tangent bundle is Ulrich are the twisted cubic and the Veronese surface. Moreover, we prove that the cotangent bundle is never Ulrich.
\end{abstract}

\maketitle

%\printindex

\setcounter{tocdepth}{1}
\tableofcontents

\section{Introduction}

Ulrich vector bundles were introduced in \cite{Ulr84} in terms of commutative algebra. Their first geometric manifestations go back to seminal works such as \cite{Bea00,ES03}, where the authors relate in a very precise way the existence of such bundles on a smooth projective complex variety $X\subseteq \mathbb{P}^N$ to some measurement of the complexity of the underlying polarized variety $(X,\mathcal{O}_X(1))$. We refer the reader to \cite{Cos17,Bea18} for comprehensive introductions to Ulrich bundles.

One of the major questions of the subject is whether or not every smooth projective variety admits an Ulrich bundle with respect to some polarization. This question has been answered positively in many interesting situations. For instance, we have the following (non-exhaustive) list of smooth projective complex varieties that are known to support Ulrich bundles:

\begin{itemize}
    \item Algebraic curves. See \cite[\S 4]{ES03}.
    \item Complete intersections in the projective space. See \cite{BHU87,HUB91}.
    \item Veronese varieties. See \cite[\S 5]{ES03}.
    \item Grassmannians and, more generally, many rational homogeneous spaces of Picard number one carry an equivariant Ulrich bundle. See \cite{CMR15,Fon16,LP21}.
    \item Some ruled and del Pezzo surfaces. See \cite{ES03,Cas17,ACMR18,Cas19,ACCMRT20}.
    \item Abelian surfaces, bielliptic surfaces, K3 surfaces and Enriques surfaces. See \cite{Bea16,AFO17,Cas17,Bea18,BN18,Fae19}.
    \item Elliptic and some quasi-elliptic surfaces. See \cite{MRPL19,Lop21}.
    \item Some surfaces of general type. See \cite{Cas17,Bea18,Cas18,Cas19,Lop19,Lop21}.
    \item Some Fano threefolds of Picard number one. See \cite{Bea18}.
\end{itemize}

There are several techniques that have been developed in order to produce Ulrich bundles. Among them, we can mention the use of tools from commutative algebra (cf. \cite{HUB91}), representation theory and Borel--Bott--Weil theorem (cf. \cite{ES03}), the study of some Noether--Lefschetz loci (cf. \cite{AFO17}), the Hartshorne--Serre construction (cf. \cite[\S 6]{Bea18}), and the use of the Cayley--Bacharach property for suitable zero-dimensional sub-schemes on surfaces (cf. \cite[\S 5]{Bea18}). Additionally, these techniques have been modified in order to produce certain torsion-free sheaves which, in some cases, can be deformed inside their corresponding moduli space of semi-stable sheaves to obtain the desired Ulrich vector bundles (cf. \cite{Fae19}). 

In practice, many of these constructions are not explicit, and they depend on some choices (e.g., suitable codimension two subschemes to be used in the Cayley--Bacharach or Hartshorne--Serre construction). Because of this, even for varieties where the existence of Ulrich bundles is known to be true, it is a natural and challenging problem to classify Ulrich bundles with fixed numerical invariants and to determine the {\em Ulrich complexity} of a given smooth projective variety $X$ (i.e., the minimum integer $r\in \mathbb{N}_{\geq 1}$ such that there exists a rank $r$ Ulrich bundle on $X$). We refer the interested reader to \cite{BES17, FK20} (and the references therein) for some results towards the Ulrich complexity of smooth hypersurfaces in the projective space.

In this article, we adopt the new approach to the study of Ulrich bundles that was recently initiated by Lopez and his collaborators in \cite{Lop20,LM21,LS21}. More precisely, they study the positivity properties of Ulrich vector bundles and give classification results for projective varieties carrying Ulrich bundles for which these positivity conditions fail. Along the same lines, it is natural to try to understand (by means of positivity techniques) smooth projective manifolds that enjoy the property of having Ulrich vector bundles that are canonically attached to them.

The main results in this article state that if $X\subseteq \mathbb{P}^N$ is a smooth projective variety of dimension $n\geq 1$, then the cotangent bundle $\Omega_X^1$ is never an Ulrich bundle (see Theorem \ref{thm:theefolds}) and for the tangent bundle we have the following:

\begin{mainthm}\label{mainthm}\hypertarget{maintheorem}{}
Let $(X,\mathcal{O}_X(H))$ be a smooth projective polarized variety of dimension $n\geq 1$. If $T_X$ is an Ulrich vector bundle with respect to $H$, then $(X,\mathcal{O}_X(H))$ is isomorphic to the twisted cubic in $\mathbb{P}^3$ or to the Veronese surface in $\mathbb{P}^5$. 
\end{mainthm}

Our main inspiration comes from the systematic study of the positivity of the tangent bundle, initiated by the solutions of Mori and Siu--Yau to the Hartshorne and Frankel conjectures and pushed further by many authors in order to give structure results for manifolds whose tangent bundle satisfies weaker positivity assumptions. It turns out that smooth projective varieties with Ulrich tangent bundle fit very well into this picture, and we show that they are rational homogeneous spaces with rather a large automorphism group. We refer the reader to \cite{MOSCWW} for a survey on the Campana--Peternell conjecture about projective manifolds with nef tangent bundle and the relationship with rational homogeneous spaces and, more generally, to the recent survey \cite{Mat21} for an account on the positivity of the tangent bundle from an analytic and algebraic point of view.

A first ingredient for our analysis, is that there are many numerical restrictions on the Chern classes of Ulrich vector bundles. This has already been observed in the case of surfaces (cf. \cite[\S 6]{ES03}), and used notably by Casnati (cf. \cite{Cas17,Cas19}) in order to give a numerical characterization of Ulrich bundles on surfaces. We extend this characterization to the case of threefolds in Proposition \ref{prop:threefolds}, and we observe in Lemma \ref{lemma:c1} a useful restriction concerning the first Chern class of Ulrich bundles in any dimension. 

\subsection*{Outline of the article}  In \S \ref{subsection:ulrich} and \S \ref{subsection:RC} we recall the main properties of Ulrich bundles on smooth projective varieties and the numerical characterization of rationally connected varieties by means of movable classes. In \S \ref{subsection:G/P} we summarize some results about the automorphism groups of rational homogeneous spaces of Picard number one, and we prove Lemma \ref{lemma:Aut(G/P)} concerning large automorphism groups. In \S \ref{section:chern} we revisit some known results concerning the Chern classes of Ulrich vector bundles, and we prove Proposition \ref{prop:threefolds} and Lemma \ref{lemma:c1} reported above. Finally, in \S \ref{section:TX Ulrich} we prove our \hyperlink{maintheorem}{Main Theorem} and Theorem \ref{thm:theefolds}. To do so, we carry out an analysis depending on the dimension in \S \ref{subsection:curves}, \S \ref{subsection:surfaces} and \S \ref{subsection:threefolds} to reduce the problem to analyse higher dimensional varieties of Picard rank at least two. As was suggested to us and proved in Appendix \ref{appendix} by Vladimiro Benedetti, this last case can be settled by means of Lie algebra computations.

\subsection*{Acknowledgments} We would like to sincerely thank Baohua \textsc{Fu} for answering our questions concerning rational homogeneous spaces, Shin-Ichi \textsc{Matsumura} for discussions around the positivity of the tangent bundle, and to Gianfranco \textsc{Casnati} for his insightful comments on the first version of this article. We are particularly grateful to Vladimiro \textsc{Benedetti} for providing us with very useful references concerning simple rational homogeneous spaces and for writing Appendix \ref{appendix}, which provides a crucial step to the present article. 

The first author was partially supported by Fondecyt ANID projects 11190323 and 1200502. The second author is supported by the fellowship INDAM-DP-COFUND-2015 ``INDAM Doctoral Programme in Mathematics and/or Applications Cofunded by Marie Sklodowska-Curie Actions", Grant Number 713485. The third author was partially supported by Fondecyt ANID project 3210518.

\section{Background and Preliminaries}\label{section:preliminaries}

Let $X\subseteq \PP^N$ be a smooth projective variety over $\mathbb{C}$ of $\dim(X)=n$, and let $H$ be a very ample divisor on $X$ such that $\OO_X(H)\cong \OO_X(1)$. We denote by $d:=\deg(X)=H^n\geq 1$ the degree of $X$.

\subsection{Ulrich bundles}\label{subsection:ulrich}

Given $m\in \mathbb{Z}$ and a coherent sheaf $E$ on $X$, we write $E(mH):=E\otimes \OO_X(mH)$. With this notation, we say that $E$ is {\bf initialized} if $\operatorname{H}^0(X,E)\neq 0$ and $\operatorname{H}^0(X,E(-H))=0$.

Let us recall the following result from \cite[\S 2]{ES03} (see also \cite[Theorem 2.3]{Bea18}). 

\begin{theorem}[Eisenbud--Schreyer--Weyman]\label{thm:eisenbud}
Let $E$ be a rank $r\geq 1$ vector bundle on $(X,\mathcal{O}_X(H))$. The following are equivalent conditions:
\begin{enumerate}
\item $E$ admits a linear resolution of the form
$$0 \rightarrow \OO_{\PP^N}(-N+n)^{\oplus a_{N-n}} \rightarrow \cdots \rightarrow \OO_{\PP^N}(-1)^{\oplus a_1} \rightarrow \OO_{\PP^N}^{\oplus a_0} \rightarrow E \rightarrow 0. $$
In particular, $a_0=r\deg(X)$ and $a_i=\binom{N-n}{i}a_0$   for all $i$.
\item $\operatorname{H}^i(X,E(-jH))=0$ for all $i\geq 0$ and all $j\in \{1,\ldots,n\}$.
\item $\operatorname{H}^i(X,E(-iH))=\operatorname{H}^j(X,E(-(j+1)H))=0$ for every $i>0$ and $j<n$.
\item For all finite linear projections $\pi:X\to \PP^n$, the sheaf $\pi_*E$ is the trivial sheaf $\OO_{\PP^n}^{\oplus t}$ for some $t$.
\end{enumerate}
The vector bundle $E$ is called an {\bf Ulrich bundle} with respect to $H$ if it satisfies any of these conditions. 
\end{theorem}

As a consequence of the previous result, it can be shown that Ulrich bundles enjoy several good properties (see e.g. \cite[\S 3]{Bea18}). Let us recall the most important for our purposes. 

\begin{itemize}
\item Let $E$ be a rank $r$ Ulrich vector bundle on $(X,\mathcal{O}_X(H))$. Then $E$ is {\bf aCM} with respect to $H$, i.e., $\operatorname{H}^i(X,E(jH))=0$ for all $j\in \mathbb{Z}$ and $1\leq i \leq n-1$. Moreover, $h^0(X,E)=rd$.
\item An Ulrich bundle is $0$--regular in the sense of Castelnuovo--Mumford, and hence it is globally generated (see e.g. \cite[\S 1.8.A]{LazI}). In particular, an Ulrich bundle is nef.
\item Let $E$ be an Ulrich vector bundle on $(X,\mathcal{O}_X(H))$ and $Y\in |H|$ be a smooth hyperplane section. Then $E|_Y$ is an Ulrich bundle on $Y$ with respect to $\mathcal{O}_X(H)|_Y$.
\item Let $E$ be a rank $r$ Ulrich bundle on $(X,\mathcal{O}_X(H))$. Then $E$ is $H$--{\bf semistable}, i.e., for every non-zero sub-sheaf $\mathcal{F}\subseteq E$ we have that $\mu_H(\mathcal{F})\leq \mu_H(E)$, where
$$\mu_H(\mathcal{F}):=\frac{c_1(\mathcal{F})\cdot H^{n-1}}{\operatorname{rk}(\mathcal{F})}\in \mathbb{Q}.$$
This follows from Theorem \ref{thm:eisenbud}(4) (see also \cite[Theorem 2.9]{CHGS12}).
\end{itemize}

Finally, we mention that a nice class of Ulrich bundle are the so-called \emph{special} Ulrich bundles.

\begin{definition}\label{defi:special}
Let $(X,\mathcal{O}_X(H))$ be a smooth projective variety of dimension $n$. A rank 2 vector bundle $E$ on $X$ is called a {\bf special Ulrich bundle} (or Ulrich special) if it is an Ulrich bundle with respect to $H$ and $\det(E) = \omega_X\otimes \mathcal{O}_X((n+1)H)$.
\end{definition}

\subsection{Numerical characterization of uniruled and rationally connected varieties} \label{subsection:RC}
We refer the reader to \cite[Ch. 4]{Deb01} for an introduction to uniruled and rationally connected varieties, as well as their main properties.

We recall the main results concerning the semi-stability of sheaves with respect to a movable curve class, a notion introduced in \cite{CP11} and further developped in \cite{GKP16}. Let $\operatorname{N}_1(X)_\mathbb{R}$ be the real vector space of numerical curve classes on $X$.

\begin{definition} A curve class $\alpha\in \operatorname{N}_1(X)_\mathbb{R}$ is called {\bf movable} if $D\cdot \alpha \geq 0$ for every effective Cartier divisor $D$ on $X$. The set of movable classes form a closed convex cone $\operatorname{Mov}_1(X)\subseteq \operatorname{N}_1(X)_\mathbb{R}$, called the {\bf movable cone} of $X$.
\end{definition}

\begin{remark}
Since $X$ is smooth and projective, it follows from \cite{BDPP13} that $\operatorname{Mov}_1(X)$ is the closure of the convex cone in $\operatorname{N}_1(X)_\mathbb{R}$ generated by the classes of curves whose deformations cover a dense subset of $X$. Moreover, a numerical divisor class $[D]\in \operatorname{N}^1(X)_\mathbb{R}$ is pseudo-effective if and only if $D\cdot \alpha \geq 0$ for all $\alpha \in \operatorname{Mov}_1(X)$.
\end{remark}

Let $\mathcal{F}$ be a non-zero torsion-free coherent sheaf on $X$. Recall that $c_1(\mathcal{F})=(\bigwedge^r \mathcal{F})^{\vee \vee}$, where $r=\operatorname{rk}(\mathcal{F})\geq 1$ is the (generic) rank of $\mathcal{F}$. The {\bf slope} of $\mathcal{F}$ with respect to a movable curve class $\alpha\in \operatorname{Mov}_1(X)$ is defined by
$$ \mu_\alpha(\mathcal{F}):=\frac{c_1(\mathcal{F})\cdot \alpha}{\operatorname{rk}(\mathcal{F})}\in \mathbb{R}.$$
As before, we say that $\mathcal{F}$ is $\alpha$--{\bf semistable} if, for every non-zero subsheaf $\mathcal{G}\subseteq \mathcal{F}$, we have that $\mu_\alpha(\mathcal{G})\leq \mu_\alpha(\mathcal{F})$. 

As it was already observed in \cite{CP11} (see also \cite{GKP16}), many of the properties of classical slope semi-stability extend to this setting. For instance, the following quantities 
\begin{align*}
    \mu_\alpha^{\max}(\mathcal{F})&:=\sup\left\{\mu_\alpha(\mathcal{G}),\;\mathcal{G}\subseteq \mathcal{F} \textup{ non-zero coherent subsheaf} \right\},\\
    \mu_\alpha^{\min}(\mathcal{F})&:=\inf\left\{\mu_\alpha(\mathcal{Q}),\; \mathcal{F} \twoheadrightarrow \mathcal{Q} \textup{ non-zero torsion-free quotient} \right\},
\end{align*}
are finite, they satisfy $\mu_\alpha^{\max}(\mathcal{F})=-\mu_\alpha^{\min}(\mathcal{F}^\vee)$, and they can be computed by means of the {\bf Harder--Narasimhan filtration} of $\mathcal{F}$ with respect to $\alpha$. Namely, there exists a unique filtration
$$\operatorname{HN}_\bullet^\alpha(\mathcal{F}):\quad 0=\mathcal{F}_0 \subsetneq \mathcal{F}_1\subsetneq \cdots \subsetneq \mathcal{F}_\ell = \mathcal{F},$$
where each quotient $\mathcal{Q}_i:=\mathcal{F}_i\slash \mathcal{F}_{i-1}$ is $\alpha$--semistable, and $\mu_\alpha^{\max}(\mathcal{F})=\mu_\alpha(\mathcal{Q}_1)>\mu_\alpha(\mathcal{Q}_2)>\cdots >\mu_\alpha(\mathcal{Q}_\ell)=\mu_\alpha^{\min}(\mathcal{F})$. In particular, $\mathcal{F}$ is $\alpha$--semistable if and only if $\mu_\alpha^{\max}(\mathcal{F})=\mu_\alpha^{\min}(\mathcal{F})$.

Using the above notation, we can state the following remarkable results by Boucksom, Demailly, P\u{a}un and Peternell in \cite[Theorem 2.6]{BDPP13} and by Campana and P\u{a}un in \cite[Theorem 4.7]{CP19} (see also \cite[\S 1.5]{Cla17}).

\begin{theorem}\label{thm:BDPP+CP}
Let $X$ be a smooth projective manifold. Then
\begin{enumerate}
    \item $X$ is uniruled if and only if there exists $\alpha\in \operatorname{Mov}_1(X)$ such that $\mu_\alpha^{\max}(T_X)>0$;
    \item $X$ is rationally connected if and only if there exists $\alpha\in \operatorname{Mov}_1(X)$ such that $\mu_\alpha^{\min}(T_X)>0$.
\end{enumerate}
\end{theorem}

\subsection{Rational homogeneous spaces with large automorphism groups}\label{subsection:G/P} \hfill

Recall that a {\bf rational homogeneous space} is a projective manifold $X$ given by the quotient $G/P$ of a semisimple complex Lie group $G$ and a parabolic subgroup $P\subseteq G$. In particular, if $G/P$ has Picard number one, then it follows that $G$ is a {\em simple} complex Lie group and $P$ is a {\em maximal} parabolic subgroup of $G$ (see e.g. \cite[\S 7.4.1]{Tev05}).

These manifolds can be classified in terms of the associated simple complex Lie algebra $\mathfrak{g}$ together with the marking of a single node in its corresponding Dynkin diagram. The dimension of the manifold $G/P_r$, where $P_r$ denotes the parabolic subgroup associated to the $r$-th node of the Dynkin diagram of $\mathfrak{g}$, can be found in \cite[\S 9.3]{Sno89} (where the simple roots are ordered according to Tits; we follow Bourbaki's ordering instead). 

We summarize the relevant information for us in Table \ref{table:G/P} below, and we refer the reader to \cite[Table 2]{MOSCWW} for the geometric description of each manifold $G/P_r$.

\begin{table}[h]
\begin{center}
\begin{tabular}{|c|c|c|c|c|c|c|c|c|c|c|c|}
\hline
Lie algebra $\mathfrak{g}$ & Dynkin diagram & $\dim_{\mathbb{C}}\mathfrak{g}$& \multicolumn{9}{c|}{$n=\dim_{\mathbb{C}}(G/P_r)$} \\ \hline
\multirow{2}{*}{$A_\ell\; (\ell\geq 1)$} & \multirow{2}{*}{\dynkin[labels={1,2,\ell-1,\ell},scale=2,mark=o]A{}} & \multirow{2}{*}{$\ell^2+2\ell$} & \multicolumn{9}{c|}{\multirow{2}{*}{$r(\ell+1-r)$}} \\ 
& & & \multicolumn{9}{c|}{}\\ \hline
\multirow{2}{*}{$B_\ell\; (\ell\geq 2)$} & \multirow{2}{*}{\dynkin[labels={1,2,\ell-2,\ell-1,\ell},scale=2,mark=o]B{}} & \multirow{2}{*}{$2\ell^2+\ell$} & \multicolumn{9}{c|}{\multirow{2}{*}{$\dfrac{r}{2}(4\ell+1-3r)$}} \\ 
& & & \multicolumn{9}{c|}{}\\ \hline
\multirow{2}{*}{$C_\ell\; (\ell\geq 3)$} & \multirow{2}{*}{\dynkin[labels={1,2,\ell-2,\ell-1,\ell},scale=2,mark=o]C{}} & \multirow{2}{*}{$2\ell^2+\ell$} & \multicolumn{9}{c|}{\multirow{2}{*}{$\dfrac{r}{2}(4\ell+1-3r)$}} \\ 
& & & \multicolumn{9}{c|}{}\\ \hline
\multirow{3}{*}{$D_\ell\; (\ell\geq 4)$} & \multirow{3}{*}{\dynkin[labels={1,2,\ell-3,\ell-2,\ell-1,\ell},label directions={,,,right,right,right},mark=o,scale=2,edgelength=2.5mm]D{}} & \multirow{3}{*}{$2\ell^2-\ell$} & \multicolumn{9}{c|}{\multirow{3}{*}{$\dfrac{r}{2}(4\ell-1-3r)$}} \\ 
& & &  \multicolumn{9}{c|}{} \\ 
& & & \multicolumn{9}{c|}{}\\ \hline
\multirow{2}{*}{$E_6$} & \multirow{2}{*}{\dynkin[mark=o,label,ordering=Bourbaki]E6} & \multirow{2}{*}{78} & $r$ & 1 & 2 & 3 & 4 & 5 & 6 &  \multicolumn{2}{c|}{} \\ 
 & & & $n$ & 16 & 21 & 25 & 29 & 25 & 16 &  \multicolumn{2}{c|}{} \\\hline
\multirow{2}{*}{$E_7$} & \multirow{2}{*}{\dynkin[mark=o,label,ordering=Bourbaki]E7} & \multirow{2}{*}{133} & $r$ & 1 & 2 & 3 & 4 & 5 & 6 & 7 &  \\
 & & & $n$ & 33 & 42 & 47 & 53 & 50 & 42 & 27 & \\ \hline
\multirow{2}{*}{$E_8$} & \multirow{2}{*}{\dynkin[mark=o,label,ordering=Bourbaki]E8} & \multirow{2}{*}{248}  & $r$ & 1 & 2 & 3 & 4 & 5 & 6 & 7 & 8 \\
 & & & $n$ & 78 & 92 & 98 & 106 & 104 & 97 & 83 & 57 \\ \hline
\multirow{2}{*}{$F_4$} & \multirow{2}{*}{\dynkin[mark=o,label,ordering=Bourbaki]F4} & \multirow{2}{*}{52}  & $r$ & 1 & 2 & 3 & 4 & \multicolumn{4}{c|}{}  \\
 & & & $n$ & 15 & 20 & 20 & 15 & \multicolumn{4}{c|}{} \\ \hline
\multirow{2}{*}{$G_2$} & \multirow{2}{*}{\dynkin[mark=o,label,ordering=Bourbaki]G2} & \multirow{2}{*}{14}  & $r$ & 1 & 2 & \multicolumn{6}{c|}{} \\
 & & & $n$ & 5 & 5 & \multicolumn{6}{c|}{} \\ \hline
\end{tabular}
\end{center}
\caption{Rational homogeneous spaces of Picard number one.}
\label{table:G/P}
\end{table}

Note that if $X$ is a smooth projective variety, then $\operatorname{Lie}(\operatorname{Aut}^\circ(X))\cong \operatorname{H}^0(X,T_X)$, where $\operatorname{Aut}^\circ(X)$ is the connected component of the identity in $\operatorname{Aut}(X)$. 

The automorphism groups of rational homogeneous manifolds $G\slash P$, where $G$ is a simple complex Lie group, have been extensively studied. Following Demazure \cite{Dem77}, a pair $(G,P)$ is non-exceptional if $\operatorname{Aut}^\circ(G/P) \cong G$. The exceptional cases (i.e., for which there is a different pair $(G',P')$ such that $G'\slash P' \cong G\slash P$) are well-known: they correspond geometrically to the odd-dimensional projective space $\mathbb{P}^{2\ell -1}$, the Spinor variety $\mathbb{S}_\ell$, and the smooth quadric hypersurface $\mathbb{Q}^5\subseteq \mathbb{P}^6$ (see e.g. \cite[Footnote 6]{Tit62} and \cite[\S 2]{Dem77}).

\begin{lemma}\label{lemma:Aut(G/P)}
Let $X\cong G\slash P$ be a rational homogeneous space of Picard number one and dimension $n$. Then we have
\[
\dim \operatorname{H}^0(X,T_X)\geq \frac{n(n+2)}{2}
\]
if and only if $X$ is isomorphic to $\mathbb{P}^n$, $\mathbb{Q}^n$ or $\operatorname{Gr}(2,5)$.
\end{lemma}

\begin{proof}
Following the notation of Table \ref{table:G/P}, a straightforward case-by-case analysis shows that if $G$ is a classical Lie group of type
\begin{itemize}
    \item $A_\ell$, then the parabolic subgroup $P_r$ is associated to the node $r=1$ or $r=\ell$ (i.e., $X\cong \mathbb{P}^\ell$), unless $(r,\ell)\in \{(2,3),(2,4),(3,4)\}$. The latter cases correspond to $\operatorname{Gr}(2,4)\cong \mathbb{Q}^4\subseteq \mathbb{P}^5$ and $\operatorname{Gr}(2,5)\cong \operatorname{Gr}(3,5)$. 
    \item $B_\ell$, then the parabolic subgroup $P_r$ is associated to the node $r=1$ (i.e., $X\cong \mathbb{Q}^{2\ell}$), unless $(r,\ell)=(2,2)$. The latter case corresponds to $\mathbb{S}_2\cong \mathbb{P}^3$ (cf. \cite[Examples 2.1.9 (a)]{IP99}).
    \item $C_\ell$, then the parabolic subgroup $P_r$ is associated to the node $r=1$ (i.e., $X\cong \mathbb{P}^{2\ell-1}$).
    \item $D_\ell$, then the parabolic subgroup $P_r$ is associated to the node $r=1$ (i.e., $X\cong \mathbb{Q}^{2\ell-1}$), unless $(r,\ell)=(4,4)$. The latter case corresponds to $\mathbb{S}_3\cong \mathbb{Q}^6\subseteq \mathbb{P}^7$ (cf. \cite[Examples 2.1.9 (b)]{IP99}).
\end{itemize}
Similarly, we note that $G$ cannot be of type $E_\ell$, $F_4$ or $G_2$. 
\end{proof}

\begin{remark}
Suppose that $X$ is an $n$-dimensional Fano manifold of Picard number one (not necessarily rational homogeneous). It is expected that $\dim \operatorname{H}^0(X,T_X)\leq n^2+2n$ with equality if and only if $X\cong \mathbb{P}^n$ (see \cite[Conjecture 2]{HM05}). A positive result in this direction was recently obtained in \cite[Theorem 1.2]{FOX18} (cf. \cite[Theorem 1.3.2]{HM05}), where the authors proved that in the case where the {\em Variety of Minimal Rational Tangents} (VMRT) at a general point of $X$ is smooth irreducible and linearly non-degenerate, then $\dim \operatorname{H}^0(X,T_X)> n(n+1)\slash 2$ if and only if $X$ is isomorphic to $\mathbb{P}^n$, $\mathbb{Q}^n$ or $\operatorname{Gr}(2,5)$.

This condition on the VMRT holds true for rational homogeneous spaces which are {\bf irreducible Hermitian symmetric spaces} of compact type (see e.g. \cite[Main Theorem]{FH12}). These manifolds were classified by Cartan and they correspond to Grassmannians $\operatorname{Gr}(r,n)$, smooth quadric hypersurfaces $\mathbb{Q}^n \subseteq \mathbb{P}^n$, Lagrangian Grassmannians $\operatorname{Lag}(2n)$, Spinor varieties $\mathbb{S}_n$, the Cayley plane $\mathbb{OP}^2$, and the rational homogeneous space $E_7\slash P_1$ of dimension $27$.

However, as Baohua Fu kindly communicated to us, there are rational homogeneous spaces of Picard number one whose VMRT are linearly degenerate (see e.g. \cite[Table 2]{Rus12} and \cite[\S 1.4.6]{Hwa01}). We refer the interested reader to \cite{KSC06} and \cite{Hwa01} for comprehensive surveys of the theory of Varieties of Minimal Rational Tangents, developed by Hwang, Mok and Kebekus in \cite{HM99,Keb02,HM04}.
\end{remark}

\section{Chern classes of Ulrich vector bundles} \label{section:chern}

\begin{definition}\label{defi:dual}
Let $(X,\mathcal{O}_X(H))$ be a smooth projective variety of dimension $n$. For a vector bundle $E$ on $X$, we define its {\bf Ulrich dual} with respect to $H$ by
$$E^{\operatorname{ul}}:=E^\vee\otimes \mathcal{O}_X(K_X+(n+1)H). $$
In particular, it follows from Serre duality and Theorem \ref{thm:eisenbud}(3) that $E$ is an Ulrich bundle with respect to $H$ if and only if $E^{\operatorname{ul}}$ is an Ulrich bundle with respect to $H$.
\end{definition}

\begin{remark}
By Serre duality, the fact that $\operatorname{H}^0(X,E^{\operatorname{ul}}(-H))=0$ (cf. the \emph{initialized} condition) is equivalent to $\operatorname{H}^n(X,E(-nH))=0$. Moreover, since for every rank 2 vector bundle $E$ on a smooth projective variety $X$ we have that $E\cong E^\vee \otimes \det(E)$, it follows that $E$ is a special Ulrich bundle (see Definition \ref{defi:special}) if and only if $E\cong E^{\operatorname{ul}}$.
\end{remark}

With the above notation, we can restate the following characterization of Ulrich bundles on surfaces obtained by Casnati in \cite[Proposition 2.1, Corollary 2.2]{Cas17}.

\begin{proposition}[Casnati]\label{prop:casnati}
Let $(S,\mathcal{O}_S(H))$ be a smooth projective surface. For any vector bundle $E$ of rank $r$ on $S$, the following assertions are equivalent:
\begin{itemize}
\item[(a)] $E$ is an Ulrich bundle.
\item[(b)] $E$ is an aCM bundle and
\begin{equation}
    \tag{$\dag$} c_1(E)\cdot H = \frac{r}{2}(K_S+3H)\cdot H \quad\textup{and}\quad c_2(E)=\frac{1}{2}(c_1^2(E)-c_1(E)\cdot K_S)-r(H^2-\chi(S,\OO_S)). \label{eq:dag}
\end{equation}
\item[(c)] $h^0(S,E(-H))=h^0(S,E^{\operatorname{ul}}(-H))=0$ and the identities \emph{(}\ref{eq:dag}\emph{)} hold.
\end{itemize}
In particular, a rank two vector bundle $E$ on $S$ is a special Ulrich bundle if and only if $E$ is initialized and the identities
\[
\det(E)=\mathcal{O}_S(K_S+3H) \quad \textup{and} \quad c_2(E)=\frac{1}{2}(5H^2+3H\cdot K_S)+2\chi(S,\OO_S)
\]
hold.
\end{proposition}

Along the same lines, we can follow verbatim the proof of Casnati's formulas to obtain similar vanishing of certain twisted Euler characteristics of Ulrich bundles. Together with the Hirzebruch--Riemann--Roch Theorem, they give many restrictions on the Chern classes of Ulrich bundles.

Let us begin with the following observation concerning certain aCM bundles, which in the case of Ulrich bundles is a direct consequence of Theorem \ref{thm:eisenbud} (2) above.

\begin{lemma}\label{lemma:vanishing}
Let $(X,\mathcal{O}_X(H))$ be a smooth projective variety of dimension $n$. Let $E$ be an aCM bundle on $X$ with respect to $H$ such that $h^0(X,E(-H))=h^n(X,E(-nH))=0$, then
\[
\chi(X,E(-H))=\chi(X,E(-2H))=\cdots=\chi(X,E(-nH))=0.
\]
\end{lemma}

\begin{proof}
Since $E$ is an aCM vector bundle, we have that
\[
h^1(X,E(-jH))=\cdots=h^{n-1}(X,E(-jH))=0 \textup{ for }j\in \{1,\ldots,n\}.
\]
On the other hand, it follows from the short exact sequence
\[
0\longrightarrow \mathcal{O}_X(-H)\longrightarrow \mathcal{O}_X \longrightarrow \mathcal{O}_H \longrightarrow 0
\]
that $h^0(X,E(-nH))\leq h^0(X,E(-(n-1)H))\leq \cdots \leq h^0(X,E(-H))=0$. Moreover, the vanishing $h^n(X,E(-nH))=0$ implies that $\chi(X,E(-nH))=0$. 

Similarly, Serre duality and the above short exact sequence give us that
\[
h^n(X,E(-jH))=h^0(X,E^\vee(K_X+jH))\leq h^0(X,E^\vee(K_X+(j+1)H))=h^n(X,E(-(j+1)H))
\]
for every $j\in \mathbb{Z}$, and hence $h^n(X,E(-H))\leq h^n(X,E(-2H))\leq \cdots \leq h^n(X,E(-nH))=0$. We conclude that $\chi(X,E(-H))=\chi(X,E(-2H))=\cdots=\chi(X,E(-(n-1)H))=0$ as well.
\end{proof}

It is worth noting that aCM bundles satisfy the following converse statement.

\begin{lemma}\label{lemma:aCM}
Let $(X,\mathcal{O}_X(H))$ be a smooth projective variety of \emph{even} (resp. \emph{odd}) dimension $n$, and let $E$ be an aCM bundle (resp. \emph{initialized} aCM bundle) on $X$ with respect to $H$. Assume that $\chi(X,E(-H))=\chi(X,E(-nH))=0$, then $E$ is an Ulrich bundle with respect to $H$.
\end{lemma}

\begin{proof}
By definition of aCM bundle, for every $i\in \{1,\ldots,n-1\}$ we have that $\operatorname{H}^i(X,E(jH))=0$ for all $j\in \mathbb{Z}$. Therefore, in virtue of Theorem \ref{thm:eisenbud}, we are left to prove that 
\[
\operatorname{H}^0(X,E(-H))=\operatorname{H}^n(X,E(-nH))=0.
\]
If both quantities $\chi(X,E(-H))=h^0(X,E(-H))+(-1)^nh^n(X,E(-H))$ and $\chi(X,E(-nH))=h^0(X,E(-nH))+(-1)^nh^n(X,E(-nH))$ are zero, then the above vanishing conditions follow immediately when $n$ is even. On the other hand, when $n$ is odd, we obtain instead that
\[
h^0(X,E(-H))=h^n(X,E(-H)) \textup{ and }h^0(X,E(-nH))=h^n(X,E(-nH)).
\]
As already discussed during the proof of Lemma \ref{lemma:vanishing}, we have $h^0(X,E(-nH))\leq h^0(X,E(-H))$ and $h^n(X,E(-H))\leq h^n(X,E(-nH))$. In particular, if $E$ is initialized then $h^0(X,E(-H))=0$, and we deduce that $h^n(X,E(-nH))=0$ as well.
\end{proof}

It is worth mentioning that the first identity in (\ref{eq:dag}) and (\ref{eq:star}) (i.e., the one related with the \emph{first} Chern class of an Ulrich bundle) can be made more precise in some higher dimensional cases. In fact, Lopez showed in \cite[Lemma 3.2]{Lop20} that if $X$ is a smooth projective variety of dimension $n\geq 2$ such that $\operatorname{Pic}(X)\cong \mathbb{Z}$, then every rank $r$ Ulrich bundle on $(X,\mathcal{O}_X(H))$ satisfies
\[
c_1(E)=\frac{r}{2}(K_X+(n+1)H).
\]

We remark that we can drop the assumption on the Picard rank if we restrict ourselves to compute $c_1(E)\cdot H^{n-1}$ instead of $c_1(E)$.

\begin{lemma}\label{lemma:c1}
Let $(X,\mathcal{O}_X(H))$ be a smooth projective variety of dimension $n$, and let $E$ be an Ulrich bundle on $X$ with respect to $H$. Then,
\[
c_1(E)\cdot H^{n-1}=\frac{r}{2}(K_X+(n+1)H)\cdot H^{n-1}
\]
\end{lemma}

\begin{proof}
Let $H_1,\ldots,H_{n-2}$ be general members in the linear system $|H|$ and let $Y_{n-j}:=H_1\cap H_2\cap \cdots \cap H_j$ for $j\in \{1,\ldots,n-2\}$. By Bertini theorem, each $Y_j$ is smooth irreducible of dimension $j$. First, we recall that (topological) Chern classes commute with arbitrary pullback and hence $c_1(E|_{Y_j})=c_1(E)|_{Y_j}$ in $\operatorname{H}^2(Y_j,\mathbb{R})$. In particular, by indutively applying \cite[Proposition 1.8(b)]{Deb01}, we have that
\[
c_1(E|_{S})\cdot H|_{S} = c_1(E)\cdot H^{n-1}, 
\]
where $S:=Y_2$. On the other hand, we know that each $E|_{Y_j}$ is an Ulrich bundle with respect to $H|_{Y_j}$ (see \S \ref{subsection:ulrich}) and hence Casnati's formulas (\ref{eq:dag}) in Proposition \ref{prop:casnati} give
\[
c_1(E|_S)\cdot H|_S = \frac{r}{2}(K_S+3H|_S)\cdot H|_S.
\]
Finally, since $\mathcal{N}_{S/X}\cong \mathcal{O}_X(H)|_S^{\oplus (n-2)}$, we deduce from the adjunction formula and \cite[Proposition 1.8(b)]{Deb01} that $(K_S+3H|_S)\cdot H|_S = (K_X+(n+1)H)\cdot H^{n-1}$.
\end{proof}

Finally, in order to extend Casnati's characterization of Ulrich bundles on surfaces to threefolds, we recall that if $E$ is a rank $r$ vector bundle on a smooth projective threefold $X$, then the Hirzebruch--Riemann--Roch theorem takes the following form
\begin{align*}
\chi(X,E)&=\int_X \operatorname{ch}(E)\operatorname{td}(X) \\
         &=r\chi(X,\OO_X)+\frac{1}{12}c_1(E)\cdot(K_X^2 + c_2(X))+ \frac{1}{4}(2 c_2(E)-c_1^2(E))\cdot K_X \\
         &\;\;\; +\frac{1}{6}(c_1^3(E)-3c_1c_2(E)+3c_3(E)),
\end{align*}
where $\chi(X,\OO_X)=-\frac{1}{24}K_X\cdot c_2(X)$. Additionally, if $L\cong \mathcal{O}_X(D)$ is a line bundle on $X$, then
\[
c_i(E\otimes \mathcal{O}_X(D))=\sum_{j=0}^i \binom{r-i+j}{j}c_{i-j}(E)D^j \quad \textup{in }\operatorname{H}^\ast(X,\mathbb{R})
\]
for every $i\geq 0$. In particular, for $j\in \mathbb{Z}$ and $D=-jH$, we get the following relations
\begin{align*}
    c_1(E(-jH))&=c_1(E)-jrH, \\
    c_2(E(-jH))&=c_2(E)-j(r-1)c_1(E)H+j^2\frac{r(r-1)}{2}H^2, \textup{ and}\\
    c_3(E(-jH))&=c_3(E)-j(r-2)c_2(E)H+j^2\frac{(r-1)(r-2)}{2}c_1(E)H^2-j^3\frac{r(r-1)(r-2)}{6}H^3.
\end{align*}

From these formulas, we can deduce the following characterization of Ulrich bundles on smooth projective threefolds.

\begin{proposition}\label{prop:threefolds}
Let $(X,\mathcal{O}_X(H))$ be a smooth projective threefold. For any rank $r$ vector bundle $E$ on $X$, the following are equivalent:
\begin{itemize}
    \item[(a)] $E$ is an Ulrich bundle.
    \item[(b)] $E$ is an initialized aCM bundle and the identities 
        \begin{equation}
        \tag{$\star$}  \label{eq:star}
        \begin{aligned}
        & c_1(E)\cdot H^2 = \frac{r}{2}H^2(K_X+4H), \\[2mm]
        & c_2(E)\cdot H  = \frac{r}{12}(K_X^2+c_2(X)-22H^2)\cdot H+\frac{1}{2}(c_1(E)-K_X)\cdot c_1(E)H, \\[2mm]         
        & c_3(E) = c_1(E)c_2(E)-\frac{1}{3}c_1^3(E)+\frac{1}{2}(c_1^2(E)-2c_2(E))\cdot K_X \\ 
        & \qquad\qquad -\frac{1}{6}c_1(E)(K_X^2+c_2(X))+2r(H^3-\chi(X,\mathcal{O}_X)),
        \end{aligned}
        \end{equation}
    hold.
    \item[(c)] $h^0(X,E(-H))=h^1(X,E(-H))=h^1(X,E^{\operatorname{ul}}(-H))=h^1(X,E(-2H))=0$ and \emph{(}\ref{eq:star}{)} hold.
\end{itemize}
In particular, a rank two  vector bundle $E$ on $X$ is a special Ulrich bundle if and only if $h^0(X,E(-H))=h^1(X,E(-H))=h^1(X,E(-2H))=0$ and the identities
\[
\det(E) = \mathcal{O}_X(K_X + 4H) \quad \textup{and} \quad c_2(E)\cdot H = \frac{1}{6}(K_X^2+c_2(X))\cdot H + 2H^2\cdot K_X + \frac{13}{3}H^3
\]
hold.
\end{proposition}

\begin{proof}
If we assume that $E$ is Ulrich, Theorem \ref{thm:eisenbud} and Lemma \ref{lemma:vanishing} imply that $E$ is an aCM bundle and 
\[
\chi(X,E(-H))=\chi(X,E(-2H))=\chi(X,E(-3H))=0.
\]
On the other hand, it follows from the Hirzebruch--Riemann--Roch theorem and the discussion above that $\chi(X,E(-jH))=\chi(X,E)+\Delta_j$, where
\begin{align*}
\Delta_j&:=\frac{1}{12}jH\cdot (12c_2(E)+6c_1(E)K_X-6c_1^2(E)-rK_X^2-rc_2(X) )\\
& \;\;\quad +\frac{1}{4}j^2H^2\cdot(2c_1(E)-rK_X)-\frac{1}{6}j^3H^3r
\end{align*}
for every $j\in \mathbb{Z}$. In particular, the identities (\ref{eq:star}) are obtained simply by solving the linear system which is determined by the relations $\Delta_2-\Delta_1=\Delta_3-\Delta_2=\chi(X,E)+\Delta_1=0$. We conclude in this way that (a) implies (b), and that (a) implies (c).

Suppose now that $E$ is an initialized aCM bundle and that the identities (\ref{eq:star}) hold. In particular, the identities (\ref{eq:star}) implies that $\chi(X,E(-jH))=0$ for $j\in \{1,2,3\}$. The fact that $E$ is Ulrich follows therefore from Lemma \ref{lemma:aCM}. In other words, (b) implies (a).

Note that (b) implies (c), since $h^1(X,E^{\operatorname{ul}}(-H))=h^2(X,E(-3H))$ by Serre duality, and hence the vanishing of higher cohomology follows directly from the aCM assumption. 

Finally, let us check that (c) implies (a). To do so, thanks to Theorem \ref{thm:eisenbud} and our assumptions, we are left to check the vanishing conditions
\[
\operatorname{H}^2(X,E(-2H))=\operatorname{H}^3(X,E(-3H))=0.
\]
Moreover, as in the proof of Lemma \ref{lemma:vanishing}, we have that $h^0(X,E(-H))=0$ implies $h^0(X,E(-jH))=0$ for all $j\in \mathbb{N}_{\geq 1}$. As before, the identities (\ref{eq:star}) imply $\chi(X,E(-jH))=0$ for $j\in \{1,2,3\}$ and in particular we have $0 = \chi(X,E(-3H))=-h^1(X,E(-3H))-h^3(X,E(-3H))$. It follows that $h^3(X,E(-3H))=0$ and, as in the proof of Lemma \ref{lemma:vanishing}, that $h^3(X,E(-2H))=0$. The identity $\chi(X,E(-2H))=0$ implies thus $h^2(X,E(-2H))=h^1(X,E(-2H))=0$, where the last vanishing holds by assumption.
\end{proof}

\begin{example}\label{ex:Fano 3folds}
Let us illustrate a possible use of Proposition \ref{prop:threefolds} by considering the special Ulrich bundles on Fano threefolds of even index, constructed by Beauville in \cite[\S 6]{Bea18}.

Let $X$ be a smooth Fano threefold such that $-K_X=2H$ with $H$ a very ample divisor, and let us consider the associated embedding $X \subseteq \mathbb{P}^{d+1}$, where $d=\deg(X)=H^3$. If we assume that the Fano index of $X$ is exactly $2$, then it follows that $d\in \{3,4,5,6,7\}$ and all possible threefolds are classified (see \cite[Theorem 3.3.1]{IP99}).

If one tries to construct a special Ulrich bundle $E$ by means of the Cayley-Bacharach property (see e.g. \cite[Theorem 5.1.1]{HL10}) or the Hartshorne-Serre construction (see e.g. \cite[Theorem 1.1]{Arr07}), then we need to find a local complete intesection curve $\Gamma \subseteq X$ whose ideal sheaf $\mathcal{I}_\Gamma$ fits in an exact sequence of sheaves
\[
0 \longrightarrow \mathcal{L} \longrightarrow E \longrightarrow \mathcal{M}\otimes \mathcal{I}_\Gamma \longrightarrow 0,
\]
where $\mathcal{L},\mathcal{M}\in \operatorname{Pic}(X)$ are line bundles. If we consider $\Gamma$ to be the zero locus of a general section $s\in \operatorname{H}^0(X,E)$, then we have
\[
0 \longrightarrow \mathcal{O}_X \xrightarrow{\;\; \cdot s \;\;} E \longrightarrow \mathcal{O}_X(2H)\otimes \mathcal{I}_\Gamma \longrightarrow 0,
\]
since $\det(E)=\mathcal{O}_X(K_X+4H)=\mathcal{O}_X(2H)$. By the Ulrich condition $\Gamma$ is an elliptic curve (see \cite[Remark 6.3.(3)]{Bea18}). The identities (\ref{eq:star}) imply that
\[
\deg(\Gamma)=c_2(E)\cdot H = \frac{1}{6}(K_X^2+c_2(X))\cdot H + 2H^2\cdot K_X + \frac{13}{3}H^3 = \frac{1}{6}(4d+12)-4d+\frac{13}{3}d = d+2.
\]
The existence of such elliptic curve of degree $d+2$ is proved by using deformation theory in \cite[Lemma 6.2]{Bea18}. The vanishing conditions $h^0(X,E(-H))=h^1(X,E(-H))=h^1(X,E(-2H))=0$ are verified in \cite[Proposition 6.1]{Bea18}. 
\end{example}

\begin{corollary}\label{coro:K=0}
Let $(X,\mathcal{O}_X(H))$ be a smooth projective threefold with $c_1(X)=0$. Assume that $E$ is a special Ulrich bundle with respect to $H$, then
\[
12c_2(E)\cdot H - 13c_1(E)\cdot H^2 = 2c_2(X)\cdot H.
\]
In particular, if $X$ is an Abelian threefold then $12c_2(E)\cdot H = 13 c_1(E)\cdot H^2$ and the general section $s\in \operatorname{H}^0(X,E)$ defines a smooth connected curve $\Gamma=V(s)\subseteq X$ of genus $g(\Gamma)=2\deg(\Gamma)+1$.

%\vspace{2mm}
%\pedro{Aquí deberíamos poner alguna aplicación de las fórmulas anteriores al caso Ulrich especial. Siendo extremadamente optimistas: si $\kappa(X)\geq 0$, entonces $X$ no posee fibrados de Ulrich especiales.} Cuando $K_X$ trivial
%\[
%12c_2(E)\cdot H - 13c_1(E)\cdot H^2 = 2c_2(X)\cdot H
%\]
%Usando Miyaoka, el lado derecho es $\geq 0$, y luego el izquierdo también. Quizás esto es mejor que las cotas que nos da la desigualdad de Bogomolov?
%\[
%\Delta(E):=2rc_2(E)-(r-1)c_1^2(E)
%\]
%Entonces $\Delta(E)\cdot H^{n-2}\geq 0$. \pedro{Creo que no se obtiene nada adicional... (por verificar).
%}
\end{corollary}

\begin{proof}
The identity $12c_2(E)\cdot H - 13c_1(E)\cdot H^2 = 2c_2(X)\cdot H$ follows directly from Proposition \ref{prop:threefolds}, since $K_X\cdot H^2=0$ and $c_1(E)=4H$ in this case. If $X$ is an Abelian threefold then we have that $c_2(X)=0$, and hence $12c_2(E)\cdot H = 13 c_1(E)\cdot H^2$. Finally, if we consider $s\in \operatorname{H}^0(X,E)$ a general section and $\Gamma:=V(s)$, then it follows from Bertini theorem that $\Gamma$ is a smooth projective curve. We observe that in this case we have a short exact sequences of sheaves
\[
0\longrightarrow \mathcal{O}_X \xrightarrow{\;\;\cdot s \;\;} E \longrightarrow \mathcal{O}_X(4H)\otimes \mathcal{I}_{\Gamma}\longrightarrow 0,
\]
as $\det(E) = \mathcal{O}_X(4H)$ in $\operatorname{Pic}(X)$. From the twisted short exact sequence 
\[
0\longrightarrow \mathcal{O}_X(-4H) \longrightarrow E(-4H) \longrightarrow \mathcal{I}_{\Gamma}\longrightarrow 0
\]
we deduce that $h^1(X,\mathcal{I}_\Gamma)=0$ and $h^2(X,\mathcal{I}_\Gamma)=\frac{26d}{3}+h^3(X,\mathcal{I}_\Gamma)$. Indeed, the Ulrich condition implies that $h^i(X,E(-4H))=0$ for $i\leq 2$, the Kodaira vanishing implies that $h^i(X,\mathcal{O}_X(-4H))=0$ for $i\leq 2$, and $h^3(X,\mathcal{O}_X(-4H))=\frac{32d}{3}$ (resp. $h^3(X,E(-4H))=2d$) by Serre duality and Hirzebruch--Riemann--Roch (resp. since $E\cong E^{\operatorname{ul}}$ is Ulrich special). Similarly, the short exact sequence
\[
0 \longrightarrow \mathcal{I}_\Gamma \longrightarrow \mathcal{O}_X \longrightarrow \iota_\ast \mathcal{O}_\Gamma \longrightarrow 0
\]
associated to the closed embedding $\iota:\Gamma \hookrightarrow X$ implies that $h^0(\Gamma,\mathcal{O}_\Gamma)=h^3(X,\mathcal{I}_\Gamma)=1$ and $h^1(\Gamma,\mathcal{O}_\Gamma)=h^2(X,\mathcal{I}_\Gamma)$. From this, we deduce that $\Gamma$ is a connected curve of genus $g(\Gamma)=\frac{26d}{3}+1=2c_2(E)\cdot H + 1 = 2\deg(\Gamma)+1$.
\end{proof}

\section{Projective manifolds whose tangent bundle is Ulrich}\label{section:TX Ulrich}

In this section we address the main problem of the article. Namely, we would like to classify all the pairs $(X,\mathcal{O}_X(H))$ such that the tangent bundle $T_X$ (resp. the cotangent bundle $\Omega_X^1$) is an Ulrich bundle with respect to $H$, where $X\subseteq \mathbb{P}^N$ is a smooth projective variety of dimension $n$ and $H$ is a very ample divisor on $X$. As usual, we will denote by $d:=\deg(X)=H^n \geq 1$ the degree of $X$. 

We carry out an analysis depending on the dimension of $X$.

\subsection{Curves}\label{subsection:curves} \hfill

Let $E$ be a vector bundle on a smooth projective curve $(C,\mathcal{O}_C(H))$ of degree $d=\deg(H)$. It follows from Theorem \ref{thm:eisenbud} that $E$ is an Ulrich bundle if and only if 
\[
h^0(C,E(-H))=h^1(C,E(-H))=0. 
\]
Since $T_C\cong \omega_C^\vee$ and $\Omega_C^1 \cong \omega_C$ in this case, where $\omega_C$ is the canonical bundle of $C$, we deduce: 

\begin{proposition}\label{prop:curves}
Let $(C,\mathcal{O}_C(H))$ as above. Then $\Omega_C^1$ is never an Ulrich bundle, and $T_C$ is an Ulrich bundle if and only if $C$ is the {\em twisted cubic} in $\mathbb{P}^3$, i.e., $(C,\mathcal{O}_C(H))\cong (\mathbb{P}^1, \mathcal{O}_{\mathbb{P}^1}(3))$.
\end{proposition}

\begin{proof}
In the case of $\Omega_C^1 \cong \omega_C$, the Ulrich condition reduces to check the two vanishing conditions $h^0(C,\omega_C(-H))=h^1(C,\omega_C(-H))=0.$ 
They are equivalent, by Serre duality, to $h^1(C,\OO_C(H))=h^0(C,\OO_C(H))=0$. The latter vanishing is impossible since $\mathcal{O}_C(H)\cong \OO_C(1)$ is very ample, and hence $\Omega_C^1$ cannot be Ulrich with respect to any $H$.

In the case of $T_C\cong \omega_C^\vee$, we are left to check the two vanishing conditions $h^0(C,\omega_C^\vee(-H))=h^1(C,\omega_C^\vee(-H))=0$. Let $g$ be the genus of $C$. If $C\cong \PP^1$ the vanishing conditions are reduced to $h^0(\PP^1,\OO_{\PP^1}(2-d))=0$ and $h^1(\PP^1,\OO_{\PP^1}(2-d))=h^0(\PP^1,\OO_{\PP^1}(d-4))=0$, by Serre duality. The first vanishing implies that $d\geq 3$, while the second one implies that $d\leq 3$. Therefore $(C,\OO_C(H))\cong (\PP^1, \OO_{\PP^1}(3))$. On the other hand, if $g \geq 1$ then the first vanishing $h^0(C,\omega_C^\vee(-H))=0$ follows directly, since $\deg(\omega_C^\vee(-H))=2-2g-d<0$. However, the second vanishing is equivalent to $h^0(C,\omega_C^{\otimes 2}(H))=0$ by Serre duality. Since $h^1(C,\omega_C^{\otimes 2}(H))=h^0(C,\omega_C^\vee(-H))=0$, the Riemann-Roch theorem yields
\[
h^0(C,\omega_C^{\otimes 2}(H))=4g-4+d-g+1=3g-3+d\geq 1
\]
and hence $T_C$ is not an Ulrich bundle.
\end{proof}

\subsection{Surfaces}\label{subsection:surfaces} \hfill

Let us consider a smooth projective surface $S\subseteq \mathbb{P}^N$ of degree $d=H^2\geq 1$, and let $E$ be a vector bundle on $S$. Again, by Theorem \ref{thm:eisenbud}, we have that $E$ is an Ulrich bundle if and only if
\[
h^1(S,E(-H))=h^2(S,E(-2H))=h^0(S,E(-H))=h^1(S,E(-2H))=0.
\]
Before treating the general case, we consider the following motivating example.

\begin{example}\label{ex:special}
Let us recall that a vector bundle $E$ on $S$ is {\bf Ulrich special} if it is an Ulrich bundle, $\operatorname{rk}(E)=2$ and $c_1(E)=K_S+3H$. Then, the tangent bundle $T_S$ is an  Ulrich special bundle if and only if $S$ is the {\em Veronese surface} in $\PP^5$, i.e., $(S,\OO_S(H))\cong (\PP^2,\OO_{\PP^2}(2))$.

Indeed, if we assume that $T_S$ is an  Ulrich special bundle then the condition $-K_S=c_1(T_S)=3H+K_S$ in $\operatorname{NS}(S)$ implies that $-2K_S = 3H$ is very ample. In particular, $S$ is a del Pezzo surface and thus $\operatorname{NS}(S)\cong \ZZ^{10-m}$ is torsion-free, where $m=K_S^2\in \{1,\ldots,9\}$ is the anti-canonical degree of $S$. Hence, the equality $-2K_S=3H$ implies that $H=2A$ for some ample divisor $A$, and in particular $-K_S=3A$. In other words, the Fano index $i_S$ of $S$ is maximal, i.e., $S\cong \PP^2$ by the Kobayashi--Ochiai theorem \cite{KO73}. Since $\omega_{\PP^2}^\vee\cong \OO_{\PP^2}(3)$ we have that $\OO_S(H)\cong \OO_{\PP^2}(2)$ and hence $(S,\OO_S(H))\cong (\PP^2,\OO_{\PP^2}(2))$ is the Veronese surface.

We are left to check that $T_{\PP^2}$ is Ulrich with respect to $\OO_{\PP^2}(2)$. This is already stated in \cite[Proposition 5.9]{ES03} (see also \cite[Theorem 5.2]{CG17} and \cite[Example 3.1]{AHMPL19}, where the authors show moreover that $T_{\PP^2}$ is actually the \emph{unique} Ulrich bundle on the Veronese surface in $\PP^5$), but we include a short proof here for the sake of completeness.

First of all, for any smooth projective surface $S$ we have that $T_S\cong \Omega_S^1\otimes \omega_S^\vee$ and hence $T_{\PP^2}\cong \Omega_{\PP^2}^1\otimes \OO_{\PP^2}(3)$. If follows therefore that $h^1(\PP^2,T_{\PP^2}(-H))=h^1(\PP^2,\Omega_{\PP^2}^1\otimes \OO_{\PP^2}(1))=0$ by Bott vanishing. Secondly, by the same reason as before and by Serre duality, we note that 
\[
h^2(\PP^2,T_{\PP^2}(-2H))=h^2(\PP^2,T_{\PP^2}\otimes \OO_{\PP^2}(-4))=h^2(\PP^2,\Omega_{\PP^2}^1\otimes \OO_{\PP^2}(-1))=h^0(\PP^2,T_{\PP^2}\otimes \OO_{\PP^2}(-2)).
\] 
The long exact sequence in cohomology induced by the twisted Euler exact sequence
 \[
0 \rightarrow \OO_{\PP^2}(-2)\rightarrow \OO_{\PP^2}(-1)^{\oplus 3} \rightarrow T_{\PP^2}\otimes \OO_{\PP^2}(-2) \rightarrow 0
\]
implies the vanishing $h^0(\PP^2,T_{\PP^2}(-H))=h^0(\PP^2,T_{\PP^2}\otimes \OO_{\PP^2}(-2))=0$. Finally, the last condition
\[
h^1(\PP^2,T_{\PP^2}(-2H))=h^1(\PP^2,T_{\PP^2}\otimes \OO_{\PP^2}(-4))=h^1(\PP^2,\Omega_{\PP^2}^1\otimes \OO_{\PP^2}(1))=0
\]
follows from Serre duality and Bott vanishing.
\end{example}

To treat the general case, we recall the following result by Reider in \cite[Theorem 1, Remark 1.2]{Rei88}.

\begin{theorem}[Reider]\label{thm : Reider}
Let $\mathcal{L}\cong \mathcal{O}_S(D)$ be a nef line bundle on a smooth projective surface $S$. If $D^2\geq 9$, then the adjoint line bundle $\omega_S\otimes \mathcal{L}$ is {\em very ample} unless there exists a non-zero effective divisor $E$ verifying one of the following conditions:
\begin{itemize}
    \item[(a)] $D\cdot E=0$ and $E^2\in \{-1,-2\}$.
    \item[(b)] $D\cdot E=1$ and $E^2\in \{0,-1\}$.
    \item[(c)] $D\cdot E=2$ and $E^2 = 0$.
    \item[(d)] $D\cdot E=3$, $D\equiv 3E$ in $\operatorname{NS}(S)$, and $E^2=1$.
\end{itemize}
\end{theorem}

We will also need the following observation.

\begin{lemma}\label{lemma:degree}
Let $(S,\OO_S(H))$ as above. If $T_S$ is an Ulrich bundle with respect to $H$ then $\kappa(S)=-\infty$, $K_S\cdot H = -6$ and $\deg(S)=4$.
\end{lemma}

\begin{proof}
The identities (\ref{eq:dag}) in Proposition \ref{prop:casnati} imply that $c_1(T_S)\cdot H = -K_S\cdot H = 3H^2 + H\cdot K_S$ and hence $2K_S\cdot H = -3H^2 < 0$. In particular, since $H$ is very ample it follows that $K_S$ is not pseudo-effective and therefore $\kappa(S)=-\infty$ (see e.g. \cite[Lemma 14.6]{Bad01}). Finally, if follows from Bertini theorem that a general curve $C\in |H|$ is smooth irreducible and hence its genus is
\[
g(C)=1+\frac{1}{2}(H^2+K_S\cdot H)=1-\frac{1}{4}H^2 \in \ZZ_{\geq 0},
\]
from which we deduce that $\deg(S)=H^2=4$, and thus $K_S\cdot H = - 6$.
\end{proof}

We are now ready to state the main result of this subsection.

\begin{theorem}\label{thm:surfaces}
Let $(S,\mathcal{O}_S(H))$ as above. Then, the cotangent bundle $\Omega_S^1$ is never an Ulrich bundle. Moreover, the tangent bundle $T_S$ is an Ulrich bundle with respect to $H$ if and only if $S$ is the {\em Veronese surface} in $\mathbb{P}^5$, i.e., $(S,\mathcal{O}_S(H))\cong (\mathbb{P}^2, \mathcal{O}_{\mathbb{P}^2}(2))$.
\end{theorem}

\begin{proof}
First, if we assume that the cotangent bundle $\Omega_S^1$ of $S$ is an Ulrich bundle, then the identities (\ref{eq:dag}) in Proposition \ref{prop:casnati} imply that $c_1(\Omega_S^1)\cdot H = K_S\cdot H = 3H^2+K_S\cdot H$ and hence $H^2=0$, which is impossible since $H$ is very ample. Therefore, the cotangent bundle $\Omega_S^1$ is never an Ulrich bundle.

Second, if we assume that $T_S$ is an Ulrich bundle on $S$, the identities (\ref{eq:dag}) in Proposition \ref{prop:casnati} together with Lemma \ref{lemma:degree} imply that
$$\chi_{\operatorname{top}}(S)=K_S^2-8+2\chi(S,\mathcal{O}_S).$$
Combining this with Noether's formula $12\chi(S,\mathcal{O}_S)=K_S^2+\chi_{\operatorname{top}}(S)$, we get that $\chi(S,\mathcal{O}_S)=\frac{1}{5}(K_S^2-4)$.

Let us notice that the divisor $K_S+3H$ is very ample. Indeed, the divisor $D:=3H$ is very ample with $D^2=36$ and hence the fact that $K_S+D$ is very ample as well follows from Theorem \ref{thm : Reider}: since $D\cdot E\geq 3$ for every non-zero effective divisor, we only need to consider the case (d) in Reider's theorem, but in that case we would have $D^2 = 9$.

Since $K_S+3H$ is very ample, Bertini theorem implies that a general curve $C\in |K_S+3H|$ is smooth irreducible and hence 
$$g(C)=1+\frac{1}{2}(K_S^2+6K_S\cdot H+9H^2+K_S^2+3K_S\cdot H)=K_S^2-8,$$
since $H^2=4$ and $K_S\cdot H=-6$, by Lemma \ref{lemma:degree}. It follows that $K_S^2\geq 8$.

Since $\kappa(S)=-\infty$ by Lemma \ref{lemma:degree}, we know that $S$ is a ruled surface and hence birationally isomorphic to
$\Gamma\times \mathbb{P}^1$, for some smooth projective curve $\Gamma$. In particular, $p_g(S)=0$ and  $q(S)=g(\Gamma)$, from which we deduce that 
$$\chi(S,\mathcal{O}_S)=1-g(\Gamma)=\frac{1}{5}(K_S^2-4)\Leftrightarrow g(\Gamma)=\frac{1}{5}(9-K_S^2)\geq 0.$$
We conclude therefore that $K_S^2\leq 9$. By divisibility reasons, we have that $K_S^2=9$ and in particular $S$ is rational by Castelnuovo's criterion, since $q(S)=g(\Gamma)=0$ in that case.
\ 
Finally, it follows from the classification of minimal rational surfaces that the unique rational surface $S$ with $K_S^2=9$ is
$S\cong \mathbb{P}^2$. The fact that $\mathcal{O}_S(H)\cong \mathcal{
O}_{\mathbb{P}^2}(2)$ follows from $\operatorname{deg}(S)=H^2=4$.
\end{proof}

\begin{remark}\label{rema:surfaces}
It is worth mentioning that one could use a similar method as in the higher dimensional cases to prove directly that $S$ is rational, and in particular $q(S)=g(\Gamma)=0$ (cf. Theorem \ref{thm:theefolds}). However, we preferred to give an alternative proof only based on classical results for algebraic surfaces.
\end{remark}

\subsection{Threefolds and higher dimensions}\label{subsection:threefolds} \hfill

Let us first remark that we cannot expect a similar answer as in the lower dimensional cases. More precisely, we have the following observation (cf. \cite[\S 5]{ES03}).

\begin{lemma}\label{lemma:Pn}
Let $n\geq 3$ be a positive integer. Then the tangent bundle of $\PP^n$ is never Ulrich.
\end{lemma}

\begin{proof}
It follows directly from the Euler exact sequence for $T_{\PP^n}$ that $h^0(\PP^n,T_{\PP^n})=n(n+2)$. If $T_{\PP^n}$ is Ulrich with respect to $\OO_{\PP^n}(d)$ then we would have $h^0(\PP^n,T_{\PP^n})=n\deg(\PP^n)=nd^n$ and hence $d^n = n+2$. In particular, $d\geq 2$ in that case. This is impossible, since $d^n \geq 2^n > n+2$ for $n\geq 3$.
\end{proof}

The numerical characterization of rationally connected varieties discussed in \S \ref{subsection:RC} and the restrictions on the first Chern class of Ulrich bundles treated in \S \ref{section:chern} give us the following result in higher dimensions.

\begin{theorem}\label{thm:theefolds}
Let $(X,\mathcal{O}_X(H))$ be a smooth projective variety of dimension $n\geq 3$. Then,
\begin{itemize}
    \item[(a)] The cotangent bundle $\Omega_X^1$ is never an Ulrich bundle with respect to $H$.
    \item[(b)] Assume that the tangent bundle $T_X$ is an Ulrich bundle with respect to $H$, then $X\cong G/P$ is a \emph{rational} homogeneous space, where $G$ is a semi-simple complex Lie group and $P\subseteq G$ is a parabolic subgroup. In this case, $\deg(X)$ is a positive multiple of $(n+2)/\gcd(n^2+n,n+2)$.
\end{itemize}
\end{theorem}

\begin{proof}
We know by Lemma \ref{lemma:c1} that if $E$ is a rank $r$ Ulrich bundle with respect to $H$ then
\[
c_1(E)\cdot H^{n-1} = \frac{r}{2}(K_X+(n+1)H)\cdot H^{n-1}.
\]
In particular, if we assume in (a) that $\Omega_X^1$ is an Ulrich bundle, then we get that 
\[
c_1(\Omega_X^1)\cdot H^{n-1} = \frac{n}{2}(K_X+(n+1)H)\cdot H^{n-1} \Leftrightarrow \frac{n-2}{n(n+1)}K_X\cdot H^{n-1} = -H^n.\]
Since $n\geq 3$, we deduce that $K_X\cdot H^{n-1}<0$. On the other hand, the cotangent bundle $\Omega_X^1$ is $\alpha$--semistable with respect to $\alpha:=H^{n-1}$ (see \S \ref{subsection:RC}) and thus
\[
\mu_\alpha^{\max}(\Omega_X^1)=\mu_\alpha(\Omega_X^1)=\frac{K_X\cdot H^{n-1}}{n}<0,
\]
or equivalently $\mu_\alpha^{\min}(T_X)>0$, and therefore $X$ must be rationally connected by Theorem \ref{thm:BDPP+CP}. This is impossible, since in that case we would have that $\operatorname{H}^0(X,\Omega_X^1)=0$ (see e.g. \cite[Corollary 4.18]{Deb01}), which contradicts the fact that $\Omega_X^1$ is an Ulrich bundle. 

Assume now that $T_X$ is an Ulrich bundle. First of all, the fact that $X$ is a homogeneous manifold (i.e., that admits a transitive action of an algebraic group) follows from the fact that $T_X$ is globally generated (see \S \ref{subsection:ulrich}) and \cite[Proposition 2.1]{MOSCWW}. Moreover, we know by a classical result by Borel and Remmert (see \cite{BR62}) that in this case $X\cong A \times G/P$, where $A$ is an Abelian variety, $G$ is a semi-simple complex Lie group and $P\subseteq G$ is a parabolic subgroup.

In order to rule out the factor $A$, we proceed as in (a). More precisely, Lemma \ref{lemma:c1} implies that 
\[
\frac{n+2}{n(n+1)}(-K_X\cdot H^{n-1}) = H^n,
\]
and hence $\mu_{\alpha}^{\min}(T_X)=\mu_{\alpha}(T_X)>0$. It follows from Theorem \ref{thm:BDPP+CP} that $X$ is rationally connected and thus $X\cong G/P$. Moreover, the previous computation shows that $\deg(X)=H^n$ is a positive multiple of $(n+2)/\gcd(n^2+n,n+2)$, by divisibility reasons. This shows (b).
\end{proof}

\begin{remark}\label{remark:fano 3-folds} The case of threefolds with Ulrich tangent bundle can be treated using classification results.

Indeed, if $\dim(X)=3$ and $T_X$ is an Ulrich bundle with respect to $H$, then we would have by Theorem \ref{thm:theefolds} (b) that $\deg(X)\geq 5$. In particular, we would have that $\dim \operatorname{Aut}^\circ(X)=h^0(X,T_X)=\dim(X)\deg(X)\geq 15$. 

On the other hand, it is known that rational homogeneous varieties are \emph{Fano}, i.e., the anti-canonical bundle $-K_X$ is ample (see e.g. \cite[Proposition 2.3]{MOSCWW}). Using the classification of smooth Fano threefolds by Iskovskikh and Mori--Mukai (see e.g. \cite[Chapter 12]{IP99}), together with the recent results on infinite automorphism groups on Fano threefolds, we can perform a case-by-case analysis that give us the desired result. More precisely, $\dim\operatorname{Aut}^\circ(X)\leq 15$, with equality if and only if $X\cong \mathbb{P}^3$ or $X\cong \mathbb{P}(\mathcal{O}_{\mathbb{P}^2}\oplus \mathcal{O}_{\mathbb{P}^2}(2))$ by \cite[Theorem 1.1.2]{KPS18} and \cite[Theorem 1.2]{PCS19} (see also \cite[Appendix A]{BFT21}). The first case is ruled out by Lemma \ref{lemma:Pn}, while the second variety is not a rational homogeneous threefold by \cite[Theorem 6.1]{CP91}.
\end{remark}

One immediately observe that the fact of the tangent bundle being Ulrich imposes that the manifold has a rather big automorfism group. More precisely, we have the following consequence of the above result and the discussion in \S \ref{subsection:G/P}.

\begin{corollary}\label{coro:Picard one}
Let $(X,\mathcal{O}_X(H))$ be a smooth projective variety of dimension $n \geq 4$, and Picard number one. Then, the tangent bundle $T_X$ is never Ulrich.
\end{corollary}

\begin{proof}
Assume by contradiction that $T_X$ is an Ulrich bundle with respect to $H$. It follows from Theorem \ref{thm:theefolds} (b) that $d:=\deg(X)$ is a positive multiple of $\ell$, where
\[
\ell:=\left\{\begin{array}{cl}
     n+2 & \textup{ if $n$ is odd,} \\[2mm]
     \dfrac{n+2}{2} & \textup{ if $n$ is even.} 
\end{array} \right.
\]
In particular, we have that $h^0(X,T_X)=nd\geq n\ell\geq \frac{n(n+2)}{2}$. 

On the other hand, it follows from Lemma \ref{lemma:Aut(G/P)} that $X$ is isomorphic to $\mathbb{P}^n$, the smooth projective quadric hypersurface $\mathbb{Q}^n\subseteq \mathbb{P}^{n+1}$ or the Grassmannian $\operatorname{Gr}(2,5)$. Moreover, if $n$ is odd then $h^0(X,T_X)\geq n^2+2n$ and hence $X\cong \mathbb{P}^n$ in that case.

Since we know that the tangent bundle of $\mathbb{P}^n$ is not Ulrich by Lemma \ref{lemma:Pn}, we will henceforth assume that $X\cong \mathbb{Q}^{2m}\subseteq \mathbb{P}^{2m+1}$ is an even dimensional smooth quadric hypersurface or that $X\cong \operatorname{Gr}(2,5)$.

If $X\cong \mathbb{Q}^{2m}$ we have, on one hand, that $h^0(X,T_X)=\dim \mathfrak{so}_{2m+2}(\mathbb{C})=(2m+1)(m+1)$. On the other hand, the Ulrich condition and the previous discussion impose that 
\[
h^0(X,T_X)=nd=2m k\ell = 2m(m+1)k
\]
for some $k\in \mathbb{N}_{\geq 1}$, which is impossible by parity reasons.

Similarly, if $X\cong \operatorname{Gr}(2,5)$ we have, on one hand, that $h^0(X,T_X)=\dim \mathfrak{sl}_5(\mathbb{C})=24$. On the other hand, we would have that $h^0(X,T_X)=6d$ and hence $d=4$. Since $\operatorname{Pic}(\operatorname{Gr}(2,5))\cong \mathbb{Z}$ is generated by the class of $\mathcal{L}:=\varphi^\ast \mathcal{O}_{\mathbb{P}^9}(1)$, where $\varphi:\operatorname{Gr}(2,5)\hookrightarrow \mathbb{P}^9$ is the Pl\"{u}cker embedding and where $\deg(\mathcal{L})=5$, we have that $\deg(X)$ has to be a multiple of 5, which leads to a contradiction.
\end{proof}

The following question naturally arises.

\begin{question}\label{question:Pic geq 2}
Is there a rational homogeneous space $X\cong G/P$ of dimension $n\geq 4$ and Picard number at least two such that $T_X$ is an Ulrich bundle?
\end{question}

%Answering Question \ref{question:Pic geq 2} above is t
The main purpose of Appendix \ref{appendix} is to answer this question. We can prove now our main result.

  \begin{proof}[Proof of \hyperlink{maintheorem}{Main Theorem}]
By the previous results of this section, it is enough to consider $X$ to be a rational homogeneous space $G\slash P$ of dimension $\geq 4$ and Picard number $\geq 2$. We conclude from Theorem \ref{thm:Picard geq 2} below that $(X,\mathcal{O}_X(H))$ does not have Ulrich tangent bundle.
%It follows from Theorem \ref{thm:Picard geq 2} below that $X\cong (\mathbb{P}^1)^n\times (\mathbb{P}^2)^m$ for some $n,m\in \mathbb{N}$. We conclude from Proposition \ref{prop:product}.
\end{proof}

\appendix

\section{(by \textsc{Vladimiro Benedetti})}
\label{appendix}

In this appendix, we will give an answer to Question \ref{question:Pic geq 2}. %In order to do so, we will state a technical lemma about rational homogeneous projective varieties, which will prove useful in order to obtain the complete classification of varieties whose tangent bundle is Ulrich.
It has been shown in Theorem \ref{thm:theefolds} that varieties with Ulrich tangent bundle are rational homogeneous projective varieties, i.e., isomorphic to a quotient $X=G\slash P$ of a (semi)simple Lie group $G$ by a parabolic subgroup $P$. The case $\operatorname{Pic}(X)\cong \mathbb{Z}$ is treated in Corollary \ref{coro:Picard one}. Let us therefore suppose that the Picard number of $X$ is strictly greater than 1. 

%Before stating the technical lemma,
Let us recall a few basic facts about such quotients $G/P$. We denote by $\fg$ the Lie algebra of $G$ and by
\[
\fg=\fh\oplus \bigoplus_{\alpha\in \Phi}\fg_\alpha
\]
a Cartan decomposition of $\fg$, where $\fh$ is a Cartan subalgebra and $\Phi\subset \fh^\vee \cong \fh$ is the set of roots of $\fg$. Let $\Delta=\{ \alpha_1,\ldots ,\alpha_n \}$ be a basis of simple roots of $\Phi$ with Bourbaki's notation, and let $\omega_1,\ldots,\omega_n \in \fh$ be the corresponding set of fundamental weights.

Any parabolic subgroup $P=P(\Sigma)$ is uniquely defined by a subset $\Sigma \subset \Delta\cong \{ 1,\cdots,n \}$ of vertices of the Dynkin diagram associated to $G$. The Lie algebra $\fp$ of (a conjugate of) $P(\Sigma)$ is given by
\[
\fp=\fh \oplus \bigoplus_{\alpha\in \Phi^-}\fg_\alpha \oplus \bigoplus_{\alpha\in \Phi^+(\Sigma)}\fg_\alpha,
\]
where $\Phi^+(\Sigma):=\{ \alpha\in \Phi^+ \mid \alpha=\sum_{\alpha_i\notin \Sigma}c_i \alpha_i \}$.

Homogeneous vector bundles on $G/P$ are in one-to-one correspondence with representations of $P$. For any simple root $\alpha_i$ in $\Sigma$, there exists a homogeneous line bundle $L_i$ which corresponds to the one dimensional representation of $P$ whose highest weight with respect to $\fh$ is $\omega_i$. With the choices we have made, $L_i$ is a positive line bundle. The Picard group of $G/P$ is equal to
\[
\operatorname{Pic}(G/P)=\bigoplus_{i\in \Sigma}\ZZ L_i.
\]
A line bundle on $G/P$ is thus a linear combination of the bundles $L_i$. Let $L=\sum_{i\in \Sigma} a_i L_i$ be such a line bundle. It is positive (i.e. globally generated) if and only if $a_i\geq 0$ for all $i\in \Sigma$; it is ample if and only if it is very ample if and only if $a_i> 0$ for all $i\in \Sigma$. Since $-K_{G/P}$ is ample, we know that $$\det(T_{G/P})=-K_{G/P}=\sum_{i\in \Sigma} j_i L_i$$ for some integers $j_i>0$ for all $i\in \Sigma$. The technical lemma we need is the following bound on the $j_i$'s.

\begin{lemma}
\label{lemtechnical}
Let $X=G/P(\Sigma)$ be a homogeneous rational projective variety with anti-canonical bundle $-K_{X}=\sum_{i\in \Sigma} j_i L_i$. Then, $j_i<\dim(X)$ for all $i\in \Sigma$, except when $X=\mathbb{P}^n$, $X=\mathbb{Q}^{n}$ or $X=\mathbb{P}^1\times \mathbb{P}^{n-1}$.
\end{lemma}

The proof of this lemma will be essentially combinatorial, but before let us make a few remarks.

\begin{remark}
\label{rempicone}
When the Picard number of $X$ is equal to one the result is clear, since the only Fano varieties whose index $i_X$ is greater or equal than $\dim(X)$ are projective spaces or quadrics by \cite{KO73}. So we are reduced to the case when the Picard number is greater than one.
\end{remark}

\begin{remark}
When $G$ is classical, the lemma can be derived from the explicit description of $X=G\slash P$ as a flag manifold. Let us suppose that $G$ is of type $A_{n-1}$ and the Picard number of $X=G/P$ is greater than one. Then $X$ is a flag manifold $X=\Fl(i_1,\ldots,i_k,n)$, where $\Sigma=\{i_1,\ldots , i_k\}$. The line bundle $L_h$ is the determinant $L_h=\det(\mathcal{U}_h^\vee)$ of the dual of the tautological bundle of rank $i_h$. It is easy to deduce the explicit formula for the anti-canonical bundle:
\[
-K_X=L_1^{i_2}\otimes L_2^{i_3-i_1}\otimes\cdots \otimes L_k^{n-i_{k-1}}.
\]
Since $\dim(X)\geq n$, it is straightforward to check that $j_h=i_{k+1}-i_{k-1}$ is strictly smaller than $\dim(X)$ for $h\in\{1,\ldots,k\}$. 

For the other classical groups, one could proceed similarly using the fact that the corresponding homogeneous varieties $X=G/P$ are zero loci of a general section of $\wedge^2 \mathcal{U}_k^\vee$ (type $C_m$) or $S^2 \mathcal{U}_k^\vee$ (type $B_m$ and $D_m$) inside $\Fl(i_1,\ldots,i_k,n)$, which allows us to use adjunction in order to understand $-K_X$. However, this strategy does not generalize straightforwardly to the exceptional groups.
\end{remark}

\begin{proof}
Let us assume that $X=G/P(\Sigma)$ is a homogeneous rational projective variety with Picard number greater than one (see Remark \ref{rempicone}). The tangent bundle of $X$ is homogeneous, and it corresponds to a $P$-representation $\mathfrak{T}$. Since the action of $G$ on $X$ is homogeneous, we get that 
$$\mathfrak{T}\cong\fp^\perp\cong \bigoplus_{\alpha\in \Phi^+,\; \alpha\notin\Phi^+(\Sigma)} \fg_\alpha,$$
where the last expression is the decomposition of $\mathfrak{T}$ in irreducible $\fh$-modules. As a result the $\fh$-weight of $\det(\mathfrak{T})$ is equal to $c_{\Sigma}:=\sum_{\alpha\in \Phi^+, \alpha\notin\Phi^+(\Sigma)}\alpha$. Since this should be a weight of a one dimensional representation of $P$ (corresponding to the line bundle $\det(T_X)=-K_X$), one can write it as
$$ c_{\Sigma}= \sum_{i\in \Sigma} j_i \omega_i,$$ 
where the $j_i$'s are the same as those appearing in the expression of $-K_X$. Note that $j_i=(c_{\Sigma},H_{\alpha_i})$, where $(\cdot,\cdot)$ is the Killing form on $\fh$ and $H_{\alpha_i}=2\frac{\alpha_i}{(\alpha_i,\alpha_i)}$ is the co-root of $\alpha_i$. Recall finally that $(\alpha_i, H_{\alpha_j})=2$ if $i=j$, while it is negative if $i\neq j$ (and strictly negative if $\alpha_i,\alpha_j$ are not orthogonal).

Let us focus our attention on one of the exponents $j_i$ for $i\in \Sigma$. Since the Picard number of $X$ is greater than one, $\dim(X)=\dim(G/P(\Sigma))> \dim(G/P(\{ i \}))$. For any positive root $\alpha\in \Phi^+$, if $\alpha\notin \Phi^+(\{i\})$ then $\alpha\notin \Phi^+(\Sigma)$. Moreover if $\alpha\in \Phi^+(\{i\})$ then $(\alpha,H_{\alpha_i})\leq 0$. Now we will distinguish two cases.

The first case is when there exists $h\in \Sigma$, $h\neq i$ which is contained in the same connected component of the Dynkin diagram of $G$ that contains the node $i$. Then one can check easily that there exists $\alpha\in\Phi^+(\{i\})$ but $\alpha\notin \Phi^+(\Sigma)$ such that $(\alpha,H_{\alpha_i})<0$. Putting everything together with the fact that $(c_{\{i\}},H_{\alpha_i})=i_{G/P(\{i\})} \leq \dim(G/P(\{i\}))+1$, we obtain that
\[
j_i=(c_\Sigma,H_{\alpha_i})<(c_{\{i\}},H_{\alpha_i})\leq \dim(G/P(\{i\}))+1 \leq \dim(G/P(\Sigma)),
\]
thus proving the inequality we wanted.

The second case is when $i$ is the only element in $\Sigma$ which is contained in its own connected component inside the Dynkin diagram of $G$. In such a situation, and contrary to what happened in the first case, we deduce that $(c_\Sigma,H_{\alpha_i})=(c_{\{i\}},H_{\alpha_i})$, which in general gives $j_i\leq \dim(G/P(\Sigma))$. Moreover we deduce that $G/P(\Sigma)=G/P(\{i\}) \times G/P(\Sigma\setminus \{i\})$. Therefore, if $\dim(G/P(\Sigma\setminus \{i\}))>1$ or if $i_{G/P(\{i\})}<\dim(G/P(\{i\}))+1$ we obtain $j_i<\dim(G/P(\Sigma))$; the two conditions are not satisfied only when $G/P(\Sigma\setminus \{i\})=\mathbb{P}^1$ and $G/P(\{i\})=\mathbb{P}^{l}$, i.e., when $X=\mathbb{P}^1 \times \mathbb{P}^{l}$.
\end{proof}

Having obtained Lemma \ref{lemtechnical}, we can deduce the main result of this appendix. Before doing so, let us deal with the only exception with Picard number greater than one obtained in Lemma \ref{lemtechnical}.

\begin{lemma}
\label{lem_2_appendix}
The tangent bundle of $\mathbb{P}^1\times \mathbb{P}^{l}$, $l\geq 1$, is never Ulrich.
\end{lemma}

\begin{proof}
Let us suppose that $T_{\mathbb{P}^1\times \mathbb{P}^{l}}$ is Ulrich with respect to $H:=\mathcal{O}_{\mathbb{P}^1}(a)\boxtimes \mathcal{O}_{\mathbb{P}^l}(b)$, with $a,b\geq 1$; notice that $-K_{\mathbb{P}^1\times \mathbb{P}^{l}}=\mathcal{O}_{\mathbb{P}^1}(2)\boxtimes \mathcal{O}_{\mathbb{P}^l}(l+1)$. By applying Lemma \ref{lemma:c1} to $T_{\mathbb{P}^1\times \mathbb{P}^{l}}$, we obtain
\begin{align*}
    0&=\frac{(l+1)(l+2)}{l+3}H^{l+1}+K_X \cdot H^{l}\\
    & =ab^l\Big(\frac{(l+1)^2(l+2)}{l+3} -\frac{a(l+1)+2b}{ab}\Big)\geq \frac{(l+1)^2(l+2)-(l+3)^2}{l+3},
\end{align*}
where the last inequality is a consequence of the fact that $a,b\geq 1$. From Theorem \ref{thm:surfaces} we can assume that $l\geq 2$; then it is easy to check that $(l+1)^2(l+2)-(l+3)^2>0$, which gives a contradiction with the equation in Lemma \ref{lemma:c1}, thus showing that the tangent bundle is not Ulrich.
\end{proof}

\begin{theorem}\label{thm:Picard geq 2}
There exists no polarized variety $(X,\mathcal{O}_X(H))$ with Picard number greater than one whose tangent bundle is Ulrich.
\end{theorem}

\begin{proof}
By Proposition \ref{prop:curves}, Theorem \ref{thm:surfaces}, Theorem \ref{thm:theefolds} and Corollary \ref{coro:Picard one}, we can suppose that $X=G/P$ is a rational homogeneous projective variety with Picard number equal to $k>1$. From Lemma \ref{lemtechnical} and Lemma \ref{lem_2_appendix} we know that 
\[
\det(T_X)=-K_X=\sum_{i}j_iL_i,
\]
with $j_i<n:=\dim(X)$ for $i=1,\ldots,k$. Let us suppose that $H=\sum_i a_i L_i$ with $a_i>0$ for $i=1,\ldots,k$. By applying Lemma \ref{lemma:c1} to the tangent bundle, we obtain $\frac{n(n+1)}{n+2}H^{n}+K_X \cdot H^{n-1}=0$.

For any $k$-partition $\lambda=(\lambda_1,\cdots,\lambda_k)$ of $n$, the coefficient of $L_1^{\lambda_1} \cdots  L_k^{\lambda_k}$ is equal to
\[
{n\choose\lambda} a_1^{\lambda_1}\cdots a_k^{\lambda_k}\left(\frac{n(n+1)}{n+2} -\sum_i \frac{\lambda_i j_i}{n a_i} \right)\geq \frac{n(n+1)}{n+2} -\sum_i \frac{\lambda_i j_i}{n}   \geq
\]
\[
 \frac{n(n+1)}{n+2} -\sum_i  \frac{\lambda_i (n-1)}{n}  = \frac{n(n+1)}{n+2} -(n-1) = \frac{2}{n+2}>0.
\]
Since for any $\lambda$, $L_1^{\lambda_1} \cdots  L_k^{\lambda_k}>0$, we deduce that $\frac{n(n+1)}{n+2}H^{n}+K_X \cdot H^{n-1}>0$, thus the equation in Lemma \ref{lemma:c1} is never satisfied; therefore there exists no rational homogeneous projective variety $G/P$ with Picard number greater than one such that its tangent bundle is Ulrich.
\end{proof}

\bibliographystyle{alpha}
\bibliography{BMPT}

\newcommand{\etalchar}[1]{$^{#1}$}
\begin{thebibliography}{MnOSC{\etalchar{+}}15}

\bibitem[ACC{\etalchar{+}}20]{ACCMRT20}
M.~Aprodu, G.~Casnati, L.~Costa, R.~M. Mir\'{o}-Roig, and M.~Teixidor I~Bigas.
\newblock Theta divisors and {U}lrich bundles on geometrically ruled surfaces.
\newblock {\em Ann. Mat. Pura Appl. (4)}, 199(1):199--216, 2020.

\bibitem[ACMR18]{ACMR18}
M.~Aprodu, L.~Costa, and R.~M. Mir\'{o}-Roig.
\newblock Ulrich bundles on ruled surfaces.
\newblock {\em J. Pure Appl. Algebra}, 222(1):131--138, 2018.

\bibitem[AFO17]{AFO17}
M.~Aprodu, G.~Farkas, and A.~Ortega.
\newblock Minimal resolutions, {C}how forms and {U}lrich bundles on {$K3$}
  surfaces.
\newblock {\em J. Reine Angew. Math.}, 730:225--249, 2017.

\bibitem[AHMPL19]{AHMPL19}
M.~Aprodu, S.~Huh, F.~Malaspina, and J.~Pons-Llopis.
\newblock Ulrich bundles on smooth projective varieties of minimal degree.
\newblock {\em Proc. Amer. Math. Soc.}, 147(12):5117--5129, 2019.

\bibitem[Arr07]{Arr07}
E.~Arrondo.
\newblock A home-made {H}artshorne-{S}erre correspondence.
\newblock {\em Rev. Mat. Complut.}, 20(2):423--443, 2007.

\bibitem[Bad01]{Bad01}
L.~Badescu.
\newblock {\em Algebraic surfaces}.
\newblock Universitext. Springer-Verlag, New York, 2001.
\newblock Translated from the 1981 Romanian original by Vladimir Ma\c{s}ek and
  revised by the author.

\bibitem[BDPP13]{BDPP13}
S.~Boucksom, J.-P. Demailly, M.~P\u{a}un, and T.~Peternell.
\newblock The pseudo-effective cone of a compact {K}\"{a}hler manifold and
  varieties of negative {K}odaira dimension.
\newblock {\em J. Algebraic Geom.}, 22(2):201--248, 2013.

\bibitem[Bea00]{Bea00}
A.~Beauville.
\newblock Determinantal hypersurfaces.
\newblock volume~48, pages 39--64. 2000.
\newblock Dedicated to William Fulton on the occasion of his 60th birthday.

\bibitem[Bea16]{Bea16}
A.~Beauville.
\newblock Ulrich bundles on abelian surfaces.
\newblock {\em Proc. Amer. Math. Soc.}, 144(11):4609--4611, 2016.

\bibitem[Bea18]{Bea18}
A.~Beauville.
\newblock An introduction to {U}lrich bundles.
\newblock {\em Eur. J. Math.}, 4(1):26--36, 2018.

\bibitem[BES17]{BES17}
M.~Bl\"{a}ser, D.~Eisenbud, and F.-O. Schreyer.
\newblock Ulrich complexity.
\newblock {\em Differential Geom. Appl.}, 55:128--145, 2017.

\bibitem[BFT21]{BFT21}
P.~Belmans, E.~Fatighenti, and F.~Tanturri.
\newblock Polyvector fields for {F}ano 3-folds.
\newblock {\em arXiv preprint arXiv:2104.07626}, 2021.

\bibitem[BHU87]{BHU87}
J.~P. Brennan, J.~Herzog, and B.~Ulrich.
\newblock Maximally generated {C}ohen-{M}acaulay modules.
\newblock {\em Math. Scand.}, 61(2):181--203, 1987.

\bibitem[BN18]{BN18}
L.~Borisov and H.~Nuer.
\newblock Ulrich bundles on {E}nriques surfaces.
\newblock {\em Int. Math. Res. Not. IMRN}, (13):4171--4189, 2018.

\bibitem[BR62]{BR62}
A.~Borel and R.~Remmert.
\newblock \"{U}ber kompakte homogene {K}\"{a}hlersche {M}annigfaltigkeiten.
\newblock {\em Math. Ann.}, 145:429--439, 1962.

\bibitem[Cas17]{Cas17}
G.~Casnati.
\newblock Special {U}lrich bundles on non-special surfaces with {$p_g=q=0$}.
\newblock {\em Internat. J. Math.}, 28(8):1750061, 18, 2017.

\bibitem[Cas18]{Cas18}
G.~Casnati.
\newblock Special {U}lrich bundles on regular surfaces with non-negative
  {K}odaira dimension.
\newblock {\em arXiv preprint arXiv:1809.08565. To appear in Manuscripta
  Mathematica}, 2018.

\bibitem[Cas19]{Cas19}
G.~Casnati.
\newblock Ulrich bundles on non-special surfaces with {$p_g=0$} and {$q=1$}.
\newblock {\em Rev. Mat. Complut.}, 32(2):559--574, 2019.

\bibitem[CG17]{CG17}
E.~Coskun and O.~Genc.
\newblock Ulrich bundles on {V}eronese surfaces.
\newblock {\em Proc. Amer. Math. Soc.}, 145(11):4687--4701, 2017.

\bibitem[CHGS12]{CHGS12}
M.~Casanellas, R.~Hartshorne, F.~Geiss, and F.-O. Schreyer.
\newblock Stable {U}lrich bundles.
\newblock {\em Internat. J. Math.}, 23(8):1250083, 50, 2012.

\bibitem[Cla17]{Cla17}
B.~Claudon.
\newblock Positivit\'{e} du cotangent logarithmique et conjecture de
  {S}hafarevich-{V}iehweg.
\newblock Number 390, pages Exp. No. 1105, 27--63. 2017.
\newblock S\'{e}minaire Bourbaki. Vol. 2015/2016. Expos\'{e}s 1104--1119.

\bibitem[CMR15]{CMR15}
L.~Costa and R.~M. Mir\'{o}-Roig.
\newblock {$GL(V)$}-invariant {U}lrich bundles on {G}rassmannians.
\newblock {\em Math. Ann.}, 361(1-2):443--457, 2015.

\bibitem[Cos17]{Cos17}
E.~Coskun.
\newblock A survey of {U}lrich bundles.
\newblock In {\em Analytic and algebraic geometry}, pages 85--106. Hindustan
  Book Agency, New Delhi, 2017.

\bibitem[CP91]{CP91}
F.~Campana and T.~Peternell.
\newblock Projective manifolds whose tangent bundles are numerically effective.
\newblock {\em Math. Ann.}, 289(1):169--187, 1991.

\bibitem[CP11]{CP11}
F.~Campana and T.~Peternell.
\newblock Geometric stability of the cotangent bundle and the universal cover
  of a projective manifold.
\newblock {\em Bull. Soc. Math. France}, 139(1):41--74, 2011.
\newblock With an appendix by Matei Toma.

\bibitem[CP19]{CP19}
F.~Campana and M.~P\u{a}un.
\newblock Foliations with positive slopes and birational stability of orbifold
  cotangent bundles.
\newblock {\em Publ. Math. Inst. Hautes \'{E}tudes Sci.}, 129:1--49, 2019.

\bibitem[Deb01]{Deb01}
O.~Debarre.
\newblock {\em Higher-dimensional algebraic geometry}.
\newblock Universitext. Springer-Verlag, New York, 2001.

\bibitem[Dem77]{Dem77}
M.~Demazure.
\newblock Automorphismes et d\'{e}formations des vari\'{e}t\'{e}s de {B}orel.
\newblock {\em Invent. Math.}, 39(2):179--186, 1977.

\bibitem[ES03]{ES03}
D.~Eisenbud and F.-O. Schreyer.
\newblock Resultants and {C}how forms via exterior syzygies.
\newblock {\em J. Amer. Math. Soc.}, 16(3):537--579, 2003.
\newblock With an appendix by J. Weyman.

\bibitem[Fae19]{Fae19}
D.~Faenzi.
\newblock Ulrich bundles on {K}3 surfaces.
\newblock {\em Algebra Number Theory}, 13(6):1443--1454, 2019.

\bibitem[FH12]{FH12}
B.~Fu and J.-M. Hwang.
\newblock Classification of non-degenerate projective varieties with non-zero
  prolongation and application to target rigidity.
\newblock {\em Invent. Math.}, 189(2):457--513, 2012.

\bibitem[FK20]{FK20}
D.~Faenzi and Y.~Kim.
\newblock Ulrich bundles on cubic fourfolds.
\newblock {\em arXiv preprint arXiv:2011.12622}, 2020.

\bibitem[Fon16]{Fon16}
A.~Fonarev.
\newblock Irreducible {U}lrich bundles on isotropic {G}rassmannians.
\newblock {\em Mosc. Math. J.}, 16(4):711--726, 2016.

\bibitem[FOX18]{FOX18}
B.~Fu, W.~Ou, and J.~Xie.
\newblock On {F}ano manifolds of {P}icard number one with big automorphism
  groups.
\newblock {\em arXiv preprint arXiv:1809.10623. To appear in Mathematical
  Research Letters}, 2018.

\bibitem[GKP16]{GKP16}
D.~Greb, S.~Kebekus, and T.~Peternell.
\newblock Movable curves and semistable sheaves.
\newblock {\em Int. Math. Res. Not. IMRN}, (2):536--570, 2016.

\bibitem[HL10]{HL10}
D.~Huybrechts and M.~Lehn.
\newblock {\em The geometry of moduli spaces of sheaves}.
\newblock Cambridge Mathematical Library. Cambridge University Press,
  Cambridge, second edition, 2010.

\bibitem[HM99]{HM99}
J.-M. Hwang and N.~Mok.
\newblock Varieties of minimal rational tangents on uniruled projective
  manifolds.
\newblock In {\em Several complex variables ({B}erkeley, {CA}, 1995--1996)},
  volume~37 of {\em Math. Sci. Res. Inst. Publ.}, pages 351--389. Cambridge
  Univ. Press, Cambridge, 1999.

\bibitem[HM04]{HM04}
J.-M. Hwang and N.~Mok.
\newblock Birationality of the tangent map for minimal rational curves.
\newblock {\em Asian J. Math.}, 8(1):51--63, 2004.

\bibitem[HM05]{HM05}
J.-M. Hwang and N.~Mok.
\newblock Prolongations of infinitesimal linear automorphisms of projective
  varieties and rigidity of rational homogeneous spaces of {P}icard number 1
  under {K}\"{a}hler deformation.
\newblock {\em Invent. Math.}, 160(3):591--645, 2005.

\bibitem[HUB91]{HUB91}
J.~Herzog, B.~Ulrich, and J.~Backelin.
\newblock Linear maximal {C}ohen-{M}acaulay modules over strict complete
  intersections.
\newblock {\em J. Pure Appl. Algebra}, 71(2-3):187--202, 1991.

\bibitem[Hwa01]{Hwa01}
J.-M. Hwang.
\newblock Geometry of minimal rational curves on {F}ano manifolds.
\newblock In {\em School on {V}anishing {T}heorems and {E}ffective {R}esults in
  {A}lgebraic {G}eometry ({T}rieste, 2000)}, volume~6 of {\em ICTP Lect.
  Notes}, pages 335--393. Abdus Salam Int. Cent. Theoret. Phys., Trieste, 2001.

\bibitem[IP99]{IP99}
V.~A. Iskovskikh and Yu.~G. Prokhorov.
\newblock Fano varieties.
\newblock In {\em Algebraic geometry, {V}}, volume~47 of {\em Encyclopaedia
  Math. Sci.}, pages 1--247. Springer, Berlin, 1999.

\bibitem[Keb02]{Keb02}
S.~Kebekus.
\newblock Families of singular rational curves.
\newblock {\em J. Algebraic Geom.}, 11(2):245--256, 2002.

\bibitem[KO73]{KO73}
S.~Kobayashi and T.~Ochiai.
\newblock Characterizations of complex projective spaces and hyperquadrics.
\newblock {\em J. Math. Kyoto Univ.}, 13:31--47, 1973.

\bibitem[KPS18]{KPS18}
A.~G. Kuznetsov, Yu.~G. Prokhorov, and C.~A. Shramov.
\newblock Hilbert schemes of lines and conics and automorphism groups of {F}ano
  threefolds.
\newblock {\em Jpn. J. Math.}, 13(1):109--185, 2018.

\bibitem[KSC06]{KSC06}
S.~Kebekus and L.~Sol\'{a}~Conde.
\newblock Existence of rational curves on algebraic varieties, minimal rational
  tangents, and applications.
\newblock In {\em Global aspects of complex geometry}, pages 359--416.
  Springer, Berlin, 2006.

\bibitem[Laz04]{LazI}
R.~Lazarsfeld.
\newblock {\em Positivity in algebraic geometry. {I}}, volume~48 of {\em
  Ergebnisse der Mathematik und ihrer Grenzgebiete. 3. Folge. A Series of
  Modern Surveys in Mathematics [Results in Mathematics and Related Areas. 3rd
  Series. A Series of Modern Surveys in Mathematics]}.
\newblock Springer-Verlag, Berlin, 2004.
\newblock Classical setting: line bundles and linear series.

\bibitem[LMn21]{LM21}
A.~F. Lopez and R.~Mu\~{n}oz.
\newblock On the classification of non-big {U}lrich vector bundles on surfaces
  and threefolds.
\newblock {\em arXiv preprint arXiv:2101.04207}, 2021.

\bibitem[Lop19]{Lop19}
A.~F. Lopez.
\newblock On the existence of {U}lrich vector bundles on some surfaces of
  maximal {A}lbanese dimension.
\newblock {\em Eur. J. Math.}, 5(3):958--963, 2019.

\bibitem[Lop20]{Lop20}
A.~F. Lopez.
\newblock On the positivity of the first {C}hern class of an {U}lrich vector
  bundle.
\newblock {\em arXiv preprint arXiv:2008.07313. To appear in Communications in
  Contemporary Mathematics}, 2020.

\bibitem[Lop21]{Lop21}
A.~F. Lopez.
\newblock On the existence of {U}lrich vector bundles on some irregular
  surfaces.
\newblock {\em Proc. Amer. Math. Soc.}, 149(1):13--26, 2021.

\bibitem[LP21]{LP21}
K.-S. Lee and K.-D. Park.
\newblock Equivariant {U}lrich bundles on exceptional homogeneous varieties.
\newblock {\em Adv. Geom.}, 21(2):187--205, 2021.

\bibitem[LS21]{LS21}
A.~F. Lopez and J.~C. Sierra.
\newblock A geometrical view of ulrich vector bundles.
\newblock {\em arXiv preprint arXiv:2105.05979}, 2021.

\bibitem[Mat21]{Mat21}
S.-I. Matsumura.
\newblock Open problems on structure of positively curved projective varieties.
\newblock {\em to appear in Annales de la Faculte des Sciences de Toulouse},
  2021.

\bibitem[MnOSC{\etalchar{+}}15]{MOSCWW}
R.~Mu\~{n}oz, G.~Occhetta, L.~Sol\'{a}~Conde, K.~Watanabe, and
  J.~Wi\'{s}niewski.
\newblock A survey on the {C}ampana-{P}eternell conjecture.
\newblock {\em Rend. Istit. Mat. Univ. Trieste}, 47:127--185, 2015.

\bibitem[MRPL19]{MRPL19}
R.~M. Mir\'{o}-Roig and J.~Pons-Llopis.
\newblock Special {U}lrich bundles on regular {W}eierstrass fibrations.
\newblock {\em Math. Z.}, 293(3-4):1431--1441, 2019.

\bibitem[PCS19]{PCS19}
V.~V. Przhiyalkovski\u{\i}, I.~A. Cheltsov, and K.~A. Shramov.
\newblock Fano threefolds with infinite automorphism groups.
\newblock {\em Izv. Ross. Akad. Nauk Ser. Mat.}, 83(4):226--280, 2019.

\bibitem[Rei88]{Rei88}
I.~Reider.
\newblock Vector bundles of rank {$2$} and linear systems on algebraic
  surfaces.
\newblock {\em Ann. of Math. (2)}, 127(2):309--316, 1988.

\bibitem[Rus12]{Rus12}
F.~Russo.
\newblock Lines on projective varieties and applications.
\newblock {\em Rend. Circ. Mat. Palermo (2)}, 61(1):47--64, 2012.

\bibitem[Sno89]{Sno89}
D.~M. Snow.
\newblock Homogeneous vector bundles.
\newblock In {\em Group actions and invariant theory ({M}ontreal, {PQ}, 1988)},
  volume~10 of {\em CMS Conf. Proc.}, pages 193--205. Amer. Math. Soc.,
  Providence, RI, 1989.

\bibitem[Tev05]{Tev05}
E.~A. Tevelev.
\newblock {\em Projective duality and homogeneous spaces}, volume 133 of {\em
  Encyclopaedia of Mathematical Sciences}.
\newblock Springer-Verlag, Berlin, 2005.
\newblock Invariant Theory and Algebraic Transformation Groups, IV.

\bibitem[Tit63]{Tit62}
J.~Tits.
\newblock Espaces homog\`enes complexes compacts.
\newblock {\em Comment. Math. Helv.}, 37:111--120, 1962/63.

\bibitem[Ulr84]{Ulr84}
B.~Ulrich.
\newblock Gorenstein rings and modules with high numbers of generators.
\newblock {\em Math. Z.}, 188(1):23--32, 1984.

\end{thebibliography}

\end{document}